\documentclass[11pt]{article}

\usepackage{hyperref}
\usepackage{caption}
\usepackage{enumerate}
\usepackage{subcaption}
\usepackage{amsmath}
\usepackage{pgf,tikz}
\usetikzlibrary{arrows}
\usepackage{amssymb}
\usepackage{amstext}
\usepackage{wrapfig}
\usepackage{amsthm}
\usepackage{a4wide}
\usepackage{color}
\usepackage{url}
\usepackage{stmaryrd}
\usepackage{graphicx}
\usepackage{algorithmic}
\usepackage{algorithm}
\usepackage[top=3.5cm,a4paper]{geometry}
\usepackage{mathtools}

\makeatletter
\DeclareRobustCommand\widecheck[1]{{\mathpalette\@widecheck{#1}}}
\def\@widecheck#1#2{%
    \setbox\z@\hbox{\m@th$#1#2$}%
    \setbox\tw@\hbox{\m@th$#1%
       \widehat{%
          \vrule\@width\z@\@height\ht\z@
          \vrule\@height\z@\@width\wd\z@}$}%
    \dp\tw@-\ht\z@
    \@tempdima\ht\z@ \advance\@tempdima2\ht\tw@ \divide\@tempdima\thr@@
    \setbox\tw@\hbox{%
       \raise\@tempdima\hbox{\scalebox{1}[-1]{\lower\@tempdima\box
\tw@}}}%
    {\ooalign{\box\tw@ \cr \box\z@}}}
\makeatother

\newcommand{\calB}{\mathcal{B}}

\newcommand{\calN}{N}

\newcommand{\calU}{\mathcal{U}}
\newcommand{\calV}{\mathcal{V}}

\newcommand{\calY}{\mathcal{Y}}

\renewcommand{\tilde}{\widetilde}

\renewcommand{\bar}{\overline}

\newcommand{\R}{{\mathbb R}}
\newcommand{\C}{{\mathbb C}}

\newtheorem{theorem}{\bf Theorem}[section]

\newtheorem{lemma}[theorem]{\bf Lemma}
\newtheorem{remark}[theorem]{\bf Remark}

\newtheorem{corollary}[theorem]{\bf Corollary}
\newtheorem{assume}[theorem]{\bf Assumption}

\newtheorem{example}[theorem]{\bf Example}

\newcommand{\ycal}{\mathcal{Y}}
\newcommand{\ycallb}{\mathcal{Y}_{\mathrm{LB}}}
\newcommand{\ycalub}{\mathcal{Y}_{\mathrm{UB}}}
\newcommand{\ycalsub}{\mathcal{Y}_{\mathrm{SUB}}}
\newcommand{\ycalv}{\mathcal{Y}_{\mathcal{V}}}

\newcommand{\lmin}{\lambda_{\mathrm{min}}}
\newcommand{\lminmu}{\lambda_{\mathrm{min}}(A(\mu))}

\newcommand{\llb}{\lambda_{\mathrm{LB}}}
\newcommand{\lub}{\lambda_{\mathrm{UB}}}
\newcommand{\llbmu}{\lambda_{\mathrm{LB}}(\mu)}

\newcommand{\lslb}{\lambda_{\mathrm{SLB}}}
\newcommand{\lsub}{\lambda_{\mathrm{SUB}}}

\newcommand{\muind}{{\mu \in D}}
\newcommand{\image}{\mathrm{im}}

\DeclareMathOperator*{\argmin}{arg\,min}
\DeclareMathOperator*{\argmax}{arg\,max}
\DeclareMathOperator{\spn}{span}

\makeatletter
\newcommand\resetsubfigs{\setcounter{sub\@captype}{0}}
\makeatother

\begin{document}

\title{Subspace acceleration for large-scale \\ parameter-dependent Hermitian eigenproblems	
\thanks{Supported by the SNF research module \textit{A Reduced Basis Approach to Large-Scale Pseudospectra Computations} within the SNF ProDoc \textit{Efficient Numerical Methods for Partial Differential Equations}.
}}

\author{ 
	Petar Sirkovi\'{c}\thanks{ANCHP, MATHICSE, EPF Lausanne, Switzerland.
    {\tt petar.sirkovic@epfl.ch}} %
	\and %
	Daniel Kressner\thanks{ANCHP, MATHICSE, EPF Lausanne, Switzerland.
    {\tt daniel.kressner@epfl.ch}} %
}

\maketitle

\begin{abstract}
This work is concerned with approximating the smallest eigenvalue
of a parameter-dependent Hermitian matrix $A(\mu)$ for many parameter values $\mu \in \R^P$.
The design of reliable and efficient algorithms for addressing this task is of importance in a
variety of applications. Most notably, it plays a crucial role in estimating the error of reduced basis methods for parametrized partial differential equations. The current state-of-the-art approach, the so called Successive Constraint Method (SCM), addresses affine linear parameter dependencies by combining sampled Rayleigh quotients with linear programming techniques.
In this work, we propose a subspace approach that additionally incorporates the sampled eigenvectors
of $A(\mu)$ and implicitly exploits their smoothness properties.
Like SCM, our approach results in rigorous
lower and upper bounds for the smallest eigenvalues on $D$. 
Theoretical and experimental evidence is given to demonstrate that our approach
represents a significant improvement over SCM in the sense that the bounds 
are often much tighter, at negligible additional cost.
\end{abstract}

\smallskip\noindent
\textbf{Keywords.} parameter-dependent eigenvalue problem, Hermitian matrix, subspace acceleration, 
Successive Constraint Method, quadratic residual bound



\section{Introduction}

Let $A:D\rightarrow \C^{\calN \times \calN}$ be a matrix-valued function on a compact subset 
$D \subset \R^P$ such that $A(\mu)$ is Hermitian for every $\mu\in D$. 
We aim at approximating the smallest eigenvalue,
\begin{equation} \label{eq:par_her_eigp}
	\lambda_{\min}( A(\mu) ),  \qquad \muind,
\end{equation}
of $A(\mu)$ for \emph{many} different values of $\muind$. We consider a large-scale setting, where applying a standard eigensolver, such as the Lanczos method~\cite{Bai2000}, is computationally feasible for a few values of $\mu$ but would become too expensive for many (e.g., thousand) parameter values.
Guiding our developments, an important application of~\eqref{eq:par_her_eigp} consists of estimating the coercivity constant for parametrized elliptic partial differential equations (PDEs), see, e.g.,~\cite{RozH08}.
In turn, these estimates can be used to construct reliable a posteriori error estimates in 
the reduced basis method (RBM) for solving such PDEs.
For more general PDEs, the coercivity constant needs to be replaced 
by the inf-sup constant, which -- after discretization -- corresponds to the smallest singular
value of a general nonsymmetric matrix $A(\mu)$ or, equivalently, to the smallest eigenvalue of the Hermitian matrix $A(\mu)^*A(\mu)$.
Other applications that require the solution of such parameter-dependent eigenvalue and singular value problems include
the computation of pseudospectra~\cite[Part IX]{Trefethen2005}, the method of 
particular solutions~\cite{Betcke2005}, and the eigenvalue analysis of waveguides~\cite{Engstrom2009}.
The related problem of optimizing the smallest eigenvalue(s) of a parameter-dependent Hermitian matrix
appears in a large variety of applications: One-parameter optimization problems play a critical role in the
design of numerical methods~\cite{Rojas2000} and robust control~\cite{Lewis1996};
multi-parameter optimization problems arise from semidefinite programming~\cite{Helmberg2000a} and graph partitioning~\cite{Kim2006,Ghosh2008}.

Without any further assumptions on the dependence of $A(\mu)$ on $\mu$, the solution of~\eqref{eq:par_her_eigp} is
computationally intractable, especially when $P$ is large.
An assumption commonly found in RBM is that $A(\mu)$ admits an affine linear decomposition
with respect to $\mu$.
\begin{assume}[Affine linear decomposition]\label{assum:aff_lin_dec} 
Given $Q \in \mathbb N$, the Hermitian matrix $A(\mu)$ admits a decomposition of the form
\begin{equation} \label{eq:affinelinear}
 	A(\mu) = \theta_1(\mu)A_1 + \dots + \theta_Q(\mu)A_Q, \quad  \forall \muind,
\end{equation}
for Hermitian matrices $A_1,\dots,A_Q \in \C^{\calN\times\calN}$ and functions $\theta_1,\dots,\theta_Q: D\mapsto \R$.
\end{assume}
Assumption~\ref{assum:aff_lin_dec} holds with small $Q$ for a number of important applications, including PDEs with parametrized coefficients on disjoint subdomains~\cite{RozH08}. 
Even when $A(\mu)$ does not satisfy this assumption, it may still be possible to approximate it very well by a short affine linear decomposition using, e.g., the Empirical Interpolation Method~\cite{Barrault2004}.

One of the simplest approaches to address~\eqref{eq:par_her_eigp} is to use Gershgorin's theorem~\cite{Johnson1989} for estimating the smallest eigenvalue, but the accuracy of the resulting estimate is usually insufficient and limits the scope of applications severely. Within the context of RBM, a number of approaches have been developed that
go beyond this simple estimate by making use of Assumption~\ref{assum:aff_lin_dec}. For example, eigenvalue  perturbation analysis can be used to 
locally approximate the smallest eigenvalues~\cite{Nguyen2005,Veroy2002}. 
The Successive Constraint Method (SCM; see \cite{Huynh2007})
is currently the most commonly used approach within RBM, probably due to its generality and relative simplicity. Variants of SCM for computing smallest singular values can be found in~\cite{Sen2006,Huynh2010}, while
an extension of SCM to non-linear problems and alternative heuristic strategies have been
proposed in~\cite{Manzoni2014}. 

If $A$ depends analytically on $\mu$ then the smallest eigenvalue inherits this property \emph{if} $\lminmu$ remains simple~\cite{Kato1995}. As shown in~\cite{Andreev2012c}, the analyticity can be used to approximate $\lminmu$ very well
by high-order Legendre polynomials (for $P = 1$) or sparse tensor products of Legendre polynomials (for $P>1$ if $D$ is a hypercube). Requiring $\lminmu$ to stay simple on the whole of $D$ is, however, a rather strong condition. In general,
there are eigenvalue crossings at which $\lminmu$ is Lipschitz continuous only; see~\cite{Mengi2014} for a recently
proposed eigenvalue optimization method that takes this piecewise regularity of $\lminmu$ into account. 
For larger $P$, keeping track of eigenvalue crossings explicitly appears to be a rather daunting task and we therefore aim
at a method for solving~\eqref{eq:par_her_eigp} that benefits only implicitly from piecewise regularity.

The approach proposed in this paper can be summarized as follows. Given $J$
parameter samples $\mu_1,\mu_2,\dots,\mu_J$, we consider the subspace $\mathcal V$ containing 
eigenvectors belonging to one or several smallest eigenvalues of $A(\mu_i)$ for $i = 1,\ldots, J$.
The smallest Ritz value of $A(\mu)$ with respect to $\mathcal V$ immediately yields an upper bound
for $\lambda_{\min}(A(\mu))$. A lower bound is obtained by combining this upper bound with a perturbation argument. To apply such an argument requires, however, knowledge on the involved eigenvalue gap. We show that this gap can be estimated by adapting the linear programming approach used in SCM. The difference between upper and lower bounds constitutes the error estimate that drives the greedy strategy for selecting
the next parameter sample $\mu_{J+1}$. The whole procedure is stopped once the error estimate is uniformly small on $D$ or, rather, on a surrogate of $D$.

As we will see in the numerical experiments, our subspace approach accelerates convergence significantly compared to SCM. Subspace approaches based on additional conditions on the parameter dependencies have been proposed in~\cite{Machiels2000,Prudhomme2002}. In the context of eigenvalue optimization problems, subspace acceleration has been discussed  in~\cite{DeVlieger2012,Kressner2014e}.

The rest of this paper is organized as follows.
In Section~\ref{sec:scm}, we first give a brief overview of SCM. We then discuss its interpolation properties and point out a limitation on the quality of the lower bounds that can possibly be attained when solely using the information taken into account by SCM.
In Section~\ref{sec:subspace}, we present our novel subspace-accelerated approach for solving~\eqref{eq:par_her_eigp}.
Furthermore, we show that the new approach has better interpolation properties than SCM.
Motivated by the fast convergence of the upper bounds in the novel approach, we also introduce residual-based lower bounds which are less reliable but sometimes converge much faster.  
In Sections~\ref{sec:examples} and~\ref{sec:infsup}, we present numerical experiments and discuss the application of our algorithm to the computation of coercivity and inf-sup constants.

\section{Successive Constraint Method} \label{sec:scm}
In the following, we recall the Successive Constraint Method (SCM)
from~\cite{Huynh2007} and derive some new properties. The basic idea
of SCM is to construct reduced-order models for~\eqref{eq:par_her_eigp} that 
allow the efficient evaluation of of lower and upper bounds for $\lminmu$. 
Being a consequence of Assumption~\ref{assum:aff_lin_dec}, the following characterization
of the smallest eigenvalue is central to SCM:
\begin{eqnarray} 
		\lminmu &=& \min_{u \in \C^{\mathcal{N}} \atop u \not=0} \frac{u^*A(\mu)u}{u^*u} \label{eq:rayleigh_qoutient}
				= \min_{u\in \C^\calN \atop u \not=0} \sum_{q=1}^Q \theta_q(\mu) \frac{u^*A_q u}{u^*u} \nonumber \\
&=& \min_{u\in \C^\calN \atop u \not=0} \theta(\mu)^TR(u)
				= \min_{y \in \calY} \theta(\mu)^Ty \label{eq:num_range_lp},
\end{eqnarray}
where we have defined the vector-valued functions 
$\theta: D \to \R^Q$, 
$R:\C^\calN\setminus \{0\}\rightarrow \R^Q$ as
\begin{equation}\label{eq:mapping_numerical_range}
	\theta(\mu):= \left[ \theta_1(\mu),\ \ldots,\ \theta_Q(\mu) \right]^T, \qquad R(u) := \left[\frac{u^*A_1u}{u^*u},\ \dots,\ \frac{u^*A_Qu}{u^*u}\right]^T,
\end{equation}
and set $\calY := \image(R)$. 
It follows from~\eqref{eq:num_range_lp} that the computation of $\lminmu$ is equivalent to optimizing a linear functional over $\calY$.
The constraint set $\calY$ is called the \emph{joint numerical range} of $A_1,\dots,A_Q$,
which is generally not convex; see \cite{Gutkin2004}. Thus standard optimization techniques cannot be used to reliably solve~\eqref{eq:num_range_lp}. 
In SCM, the set $\calY$ is approximated from above and from below
by convex polyhedra. In turn, this allows for the use of linear programming (LP) techniques to yield lower and upper bounds.

\subsection{Basic idea of SCM} \label{sec:scm_idea}

Given $J$ parameter values $C_J = \{\mu_1,\dots,\mu_J\} \subset D$, let us suppose we have computed 
the corresponding eigenpairs $(\lambda_1,v_1),\dots,(\lambda_J,v_J)$, that is, $\lambda_i$ is the smallest eigenvalue of $A(\mu_i)$ with eigenvector $v_i \in \C^\calN$.
We now describe how SCM uses this information to approximate the set $\calY$ defined above.

Clearly, 
\begin{equation}\label{eq:upper_set}
	\ycalub(C_J):= \{ R(v_i): i=1,\dots,J \}
\end{equation}
is a subset of $\calY$. Optimizing~\eqref{eq:mapping_numerical_range}
over $\ycalub(C_J)$ instead of $\calY$ thus yields an \emph{upper bound} for 
$\lminmu$. Note that this is equivalent to optimizing over the convex hull of 
$\ycalub(C_J)$, since a solution of the LP can always be attained at a vertex
of the convex polyhedron.

To get a lower bound, we first define the bounding box
\begin{equation}\label{eq:bounding_box}
	\calB = \left[ \lambda_{\mathrm{min}}(A_1), \lambda_{\mathrm{max}}(A_1)\right]\times \cdots \times 
	\left[ \lambda_{\mathrm{min}}(A_Q), \lambda_{\mathrm{max}}(A_Q)\right] \subseteq \R^Q.
\end{equation}
By the minimax characterization of eigenvalues we have 
$\calY \subset \calB$, but this approximation is often too crude and we will instead work with
\[
	\ycallb(C_J):= \{ y\in \calB : \theta(\mu_i)^Ty \geq \lambda_i, i=1,\dots,J \}.
\]
The property $\calY \subset \ycallb(C_J)$
follows from the fact that every $y = R(u_y) \in \calY$ satisfies
$\theta(\mu_i)^Ty = u_y^* A(\mu_i) u_y / u_y^* u_y \ge \min_u u^* A(\mu_i) u / u^* u = \lambda_i$. The minimax characterization also implies that the convex polyhedron 
$\ycallb(C_J)$ is tangential to $\calY$. 

With the sets defined above, we let
\begin{equation} \label{eq:minimization_bound}
	\lub(\mu;C_J) := \min_{y\in \ycalub(C_J)} \theta(\mu)^Ty, \qquad 
	\llb(\mu;C_J) := \min_{y\in \ycallb(C_J)} \theta(\mu)^Ty.
\end{equation}
Since $\ycalub(C_J) \subseteq \calY \subseteq \ycallb(C_J)$, it follows that
\[
	\llb(\mu;C_J) \leq \lminmu \leq \lub(\mu;C_J)
\]
for every $\mu \in D$. While the evaluation of $\lub(\mu;C_J)$ is trivial, the evaluation of $\llb(\mu;C_J)$ requires the solution of an LP; see Figure~\ref{fig:scm_idea}
for an illustration.  


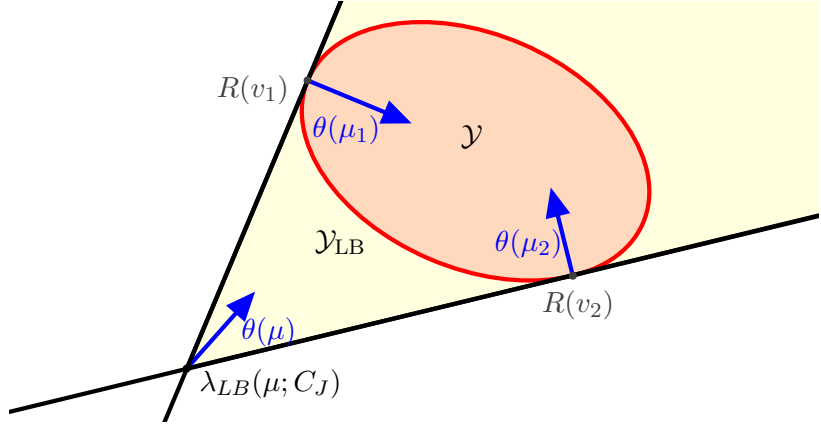
\begin{figure}[ht]
\begin{center}
\definecolor{ffffww}{rgb}{1,1,0.4}
\definecolor{xdxdff}{rgb}{0.49,0.49,1}
\definecolor{qqqqff}{rgb}{0,0,1}
\definecolor{uququq}{rgb}{0.25,0.25,0.25}
\definecolor{ffqqqq}{rgb}{1,0,0}

\begin{tikzpicture}[line cap=round,line join=round,>=triangle 45,x=1.0cm,y=1.0cm]
\clip(-3.63,-0.27) rectangle (7.01,5.29);
\fill[line width=1.6pt,color=ffffww,fill=ffffww,fill opacity=0.2] (4,13.16) -- (9.2,11.52) -- (20.29,5.68) -- (-1.3,0.43) -- cycle;
\draw [rotate around={-23.49:(2.5,3.31)},line width=1.6pt,color=ffqqqq,fill=ffqqqq,fill opacity=0.15] (2.5,3.31) ellipse (2.4cm and 1.55cm);
\draw [line width=1.6pt,domain=-3.63:7.01] plot(\x,{(--3.77--1.24*\x)/5.07});
\draw [line width=1.6pt,domain=-3.63:7.01] plot(\x,{(--5.64--3.82*\x)/1.59});
\draw [->,line width=1.6pt,color=qqqqff] (0.29,4.25) -- (1.66,3.68);
\draw [->,line width=1.6pt,color=qqqqff] (3.78,1.66) -- (3.49,2.83);
\draw (2.17,3.79) node[anchor=north west] {$\mathcal{Y}$};
\draw [->,line width=1.6pt,color=qqqqff] (-1.3,0.43) -- (-0.4,1.45);
\draw [line width=1.6pt,color=ffffww] (4,13.16)-- (9.2,11.52);
\draw [line width=1.6pt,color=ffffww] (9.2,11.52)-- (20.29,5.68);
\draw [line width=1.6pt] (20.29,5.68)-- (-1.3,0.43);
\draw [line width=1.6pt] (-1.3,0.43)-- (4,13.16);
\draw (0.25,2.4) node[anchor=north west] {$\mathcal{Y}_{\textrm{LB}}$};


\fill [color=black] (-1.3,0.43) circle (1.5pt);
\draw[color=black] (-0.2,0.23) node {$\lambda_{LB}(\mu;C_J)$};
\fill [color=uququq] (0.29,4.25) circle (1.5pt);
\draw[color=uququq] (-0.43,4.13) node {$R(v_1)$};
\fill [color=uququq] (3.78,1.66) circle (1.5pt);
\draw[color=uququq] (3.85,1.27) node {$R(v_2)$};
\draw[color=qqqqff] (0.80,3.59) node {$\theta(\mu_1)$};
\draw[color=qqqqff] (3.20,2.09) node {$\theta(\mu_2)$};
\draw[color=qqqqff] (-0.22,0.92) node {$\theta(\mu)$};
\end{tikzpicture}
\end{center}
                \caption{\label{fig:scm_idea}
                Illustration of the LP defining the lower bound $\llb(\mu;C_J)$ for $Q = 2$ and $J = 2$.
}
\end{figure}

\subsection{Error estimate and sampling strategy} \label{sec:scm_sampling}


Assessing the quality of the bounds~\eqref{eq:minimization_bound}
on the entire, usually continuous parameter domain $D$ is, in general,
an infeasible task. A common strategy in SCM, we substitute $D$ by
a training set $\Xi \subset D$ that contains finitely many (usually,
a few thousand) parameter samples. We then measure the quality
of the bounds by estimating the largest relative difference:
\begin{equation} \label{eq:error_ratio}
	\max_{\mu \in \Xi} \frac{\lub(\mu;C_J) - \llb(\mu;C_J)}{|\lub(\mu;C_J)|}.
\end{equation}

If~\eqref{eq:error_ratio} is not sufficiently small, SCM enlarges $C_J$
by a parameter that attains the maximum in~\eqref{eq:error_ratio}
and recomputes the bounds~\eqref{eq:minimization_bound}.
The resulting greedy sampling strategy is summarized in Algorithm~\ref{alg:SCM}.

\begin{algorithm}
\floatname{algorithm}{Algorithm}
\caption{Successive Constraint Method}
\begin{algorithmic}[1]
\REQUIRE Training set $\Xi$, affine linear decomposition such that  $A(\mu)=\theta_1(\mu)A_1+\dots+\theta_Q(\mu)A_Q$ is
Hermitian for every $\mu \in \Xi$. Relative error tolerance $\varepsilon_{\mathrm{SCM}}$.
\ENSURE Set $C_J \subset \Xi$ with corresponding eigenpairs $(\lambda_i,v_i)$, such that  \\ 	$\frac{\lub(\mu;C_J) - \llb(\mu;C_J)}{|\lub(\mu;C_J)|} < \varepsilon_{\mathrm{SCM}}$  for every $\mu \in \Xi$.
\STATE compute $\lambda_{\min}(A_q), \lambda_{\max}(A_q)$ for $q = 1,\ldots,Q$, defining $\calB$ according to~\eqref{eq:bounding_box}
\STATE $J=0$, $C_0=\emptyset$
\WHILE{ $\max\limits_{\mu \in \Xi} \frac{\lub(\mu;C_J) - \llb(\mu;C_J)}{|\lub(\mu;C_J)|} > \varepsilon_{\mathrm{SCM}}$  }
  \STATE $\mu_{J+1} \leftarrow \argmax\limits_{\mu \in \Xi} \frac{\lub(\mu;C_J) - \llb(\mu;C_J)}{|\lub(\mu;C_J)|} $
  \STATE $C_{J+1} \leftarrow C_J \cup \mu_{J+1}$
  \STATE recompute $\lub(\mu;C_{J+1})$ and $\llb(\mu;C_{J+1})$
  according to~\eqref{eq:minimization_bound}
  \STATE $J \leftarrow J+1$
\ENDWHILE
\end{algorithmic}
\label{alg:SCM}
\end{algorithm}

\subsection{Computational complexity} \label{sec:scm_complexity}

Let us briefly summarize the computations performed by SCM.
The bounding box $\calB$ for $\ycal$ needs to be determined initially by computing the smallest and the largest eigenvalues of 
$A_1,\dots,A_Q$. Since each iteration requires the computation of the smallest eigenpair $(\lambda_i,v_i)$ of $A(\mu_i)$,
this amounts to solving $2Q+J$ eigenproblems of size $\calN\times \calN$ in total.
Verifying the accuracy of the current approximation on $\Xi$ and selecting the next parameter sample 
requires computing $\lub(\mu;C_J)$ and $\llb(\mu;C_J)$ for all $\mu \in \Xi$. In total, this amounts to solving
$J|\Xi|$ LP problems with $Q$ variables and at most $2Q+J$ constraints.

\subsection{Interpolation results} \label{sec:scm_interpolation}

As also discussed in~\cite{Huynh2007}, it is immediate to see that the bounds produced by
SCM coincide with $\lminmu$ for all $\mu \in C_J$. The following theorem shows that the upper bounds also interpolate the derivatives of $\lminmu$ on $C_J$.

\begin{theorem}\label{thm:upper_hermite}
Let $C_J \subset D$ be finite and consider the upper bound $\lub(\mu;C_J)$ defined  in~\eqref{eq:minimization_bound}. Given $\mu_i \in C_J$ in the interior of $D$, assume that 
$\theta_1, \dots, \theta_Q: D \to \R$ are differentiable at $\mu_i$ and that $\lambda_i = \lmin(A(\mu_i))$ is a simple eigenvalue of $A(\mu_i)$. Then
	\[
		\nabla \lub (\mu_i;C_J) = \nabla \lmin (A(\mu_i)),
	\]
	with the gradient $\nabla$ with respect to $\mu$.
\end{theorem}
\begin{proof}
	Let $v_i$ be an eigenvector associated with $\lambda_i$
	such that $\|v_i\|_2 = 1$
	and set $y_i:= R(v_i) \in \ycalub(C_J)$.
	By the definition~\eqref{eq:minimization_bound}, the relation
	\begin{equation} \label{eq:charlambdamin}
	 \lub(\mu; C_J) = \min_{y\in \ycalub(C_J)} \theta(\mu)^T y = \theta(\mu)^T y_i
	\end{equation}
	holds for $\mu = \mu_i$. The simplicity of $\lambda_i$ implies that
	$y_i$ is the unique minimizer. 
	Combined with the facts that $\ycalub(C_J)$ is a discrete set, $\mu_i$ is an interior point, and $\theta(\mu)$ is continuous at $\mu_i$, this implies that~\eqref{eq:charlambdamin} also holds for all $\mu$ in a neighbourhood $\Omega \subset D$ around $\mu_i$. Consequently,
	\[
	 \frac{\partial \lub}{\partial \mu^{(p)}} (\mu_i; C_J) =  \frac{\partial \theta}{\partial \mu^{(p)}} (\mu_i)^T y_i,
	\]
	where $\mu^{(p)}$ denotes the $p$th entry of $\mu$ for $p = 1,\ldots, P$.

	On the other hand, the well-known expression for the derivative of a simple eigenvalue~\cite{Horn2013} gives 
	\begin{eqnarray*}	
		\frac{\partial \lambda_{\mathrm{min}}}{\partial \mu^{(p)}}(A(\mu_i)) &=&
		v_i^* \frac{\partial A}{\partial \mu^{(p)}}(\mu_i) v_i
				= v_i^* \Big( \sum_{q=1}^Q \frac{\partial \theta_q}{\partial \mu^{(p)}}(\mu_i)  A_q \Big) v_i \\
				&=& \sum_{q=1}^Q \frac{\partial \theta_q}{\partial \mu^{(p)}}(\mu_i) \,  v_i^* A_q v_i  
				= \frac{\partial \theta}{\partial \mu^{(p)}}(\mu_i)^T y_i,
	\end{eqnarray*}
	which completes the proof.
\end{proof}

As the following example shows, the result of Theorem~\ref{thm:upper_hermite}
does not hold for the lower bounds produced by SCM.
%
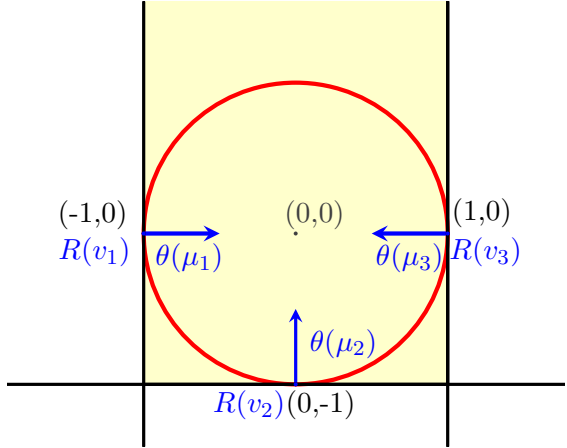
\begin{figure}[ht]
        \centering
        \definecolor{fffftt}{rgb}{1,1,0.2}
\definecolor{qqqqff}{rgb}{0,0,1}
\definecolor{ffqqqq}{rgb}{1,0,0}
\definecolor{uququq}{rgb}{0.25,0.25,0.25}
\begin{tikzpicture}[scale = 0.5, line cap=round,line join=round,>=stealth,x=1.0cm,y=1.0cm]
\clip(-7.58,-5.64) rectangle (7.07,6.14);
\fill[color=fffftt,fill=fffftt,fill opacity=0.25] (-4,10) -- (4,10) -- (4,-4) -- (-4,-4) -- cycle;
\draw [line width=1.6pt,color=ffqqqq] (0,0) circle (4cm);
\draw [line width=1.2pt] (4,-5.64) -- (4,6.14);
\draw [line width=1.2pt] (-4,-5.64) -- (-4,6.14);
\draw [line width=1.2pt,domain=-9.58:11.07] plot(\x,{(-32-0*\x)/8});
\draw [->, line width=1.6pt,color=qqqqff] (-4,0) -- (-2,0);
\draw [->,line width=1.2pt,color=qqqqff] (0,-4) -- (0,-2);
\draw [->,line width=1.6pt,color=qqqqff] (4,0) -- (2,0);
\draw [domain=-9.58:11.07] plot(\x,{(--80-0*\x)/8});
\draw (-4,10)-- (4,10);
\draw [line width=1.2pt] (4,10)-- (4,-4);
\draw [line width=1.2pt] (4,-4)-- (-4,-4);
\draw (-4,-4)-- (-4,10);
\fill [color=uququq] (0,0) circle (1.5pt);
\draw[color=uququq] (0.5,0.5) node {(0,0)};
\fill [color=black] (4,0) circle (1.5pt);
\draw[color=black] (4.9,0.55) node {(1,0)};
\fill [color=black] (-4,0) circle (1.5pt);
\draw[color=black] (-5.34,0.5) node {(-1,0)};
\draw[color=qqqqff] (-2.80,-0.65) node {$\theta(\mu_1)$};
\fill [color=black] (0,-4) circle (1.5pt);
\draw[color=black] (0.66,-4.55) node {(0,-1)};
\draw[color=qqqqff] (1.26,-2.92) node {$\theta(\mu_2)$};
\draw[color=qqqqff] (3.00,-0.65) node {$\theta(\mu_3)$};
\fill [color=qqqqff] (-4,0) circle (1.5pt);
\draw[color=qqqqff] (-5.29,-0.49) node {$R(v_1)$};
\fill [color=qqqqff] (0,-4) circle (1.5pt);
\draw[color=qqqqff] (-1.22,-4.55) node {$R(v_2)$};
\fill [color=qqqqff] (4,0) circle (1.5pt);
\draw[color=qqqqff] (5,-0.49) node {$R(v_3)$};
\end{tikzpicture}
                \caption{Joint numerical range $\ycal$ (red circle) and lower bound set $\ycallb(C_J)$ (yellow area) for the setting described in Example~\ref{ex:lb_non_interpolation}.}
                \label{fig:example_lower_bounds}
\end{figure}

\begin{example}\label{ex:lb_non_interpolation}
	For $\mu \in D:= \left[0,\pi \right]$, let 
	\[
	A(\mu) = \cos(\mu) A_1 + \sin(\mu) A_2 =  \cos(\mu) \begin{bmatrix} 1 & 0 \\ 0 & -1 \end{bmatrix} + \sin(\mu) \begin{bmatrix} 0 & -1 \\ -1 & 0 \end{bmatrix}.
	\]
It can be shown that $\ycal$, the joint numerical range of $A_1$ and $A_2$, equals the unit circle around 0.
Consider the sample set $C_3 = \{\mu_1,\mu_2,\mu_3\} = \{0,\frac{\pi}{2},\pi\}$,
	with $\lambda_{\mathrm{min}}(A(\mu_1))=\lambda_{\mathrm{min}}(A(\mu_2))=\lambda_{\mathrm{min}}(A(\mu_3))=-1$. The resulting lower bound set $\ycallb(C_3)$ is 
	the half-infinite box shown in Figure~\ref{fig:example_lower_bounds}. 
	When minimizing $\theta(\mu)^T y$ for $y \in \ycallb(C_3)$, the minimum is attained at the vertex $(-1,-1)$ for $\mu \in [\pi/4,\pi/2]$ and 
	at the vertex $(1,-1)$ for $\mu \in [\pi/2,3\pi/4]$.
	Hence,
	\[
		\llbmu = \left\{ 
		 \begin{array}{l l}
		 -\cos \mu - \sin \mu, & \text{ for } \mu \in [\pi/4,\pi/2],\\
		\cos \mu - \sin \mu, & \text{ for } \mu \in [\pi/2,3\pi/4],
		 \end{array} \right.
	\]
	yielding the following one-sided derivatives at $\mu = \pi/2$:
	\begin{eqnarray*}
		\lambda_{\mathrm{LB}}^\prime((\pi/2)^-) &=&
		\sin(\pi/2) - \cos(\pi/2)=1,\\
		\lambda_{\mathrm{LB}}^\prime((\pi/2)^+) &=& -\sin(\pi/2) - \cos(\pi/2)) =-1.
	\end{eqnarray*}
	In contrast, the exact eigenvalue is differentiable at $\pi/2$. Moreover,  $\lambda_{\mathrm{min}}^\prime(A(\pi/2)) = 0$ is different from both 
	one-sided derivatives of $\llbmu$.
\end{example}

Theorem~\ref{thm:upper_hermite} and Example~\ref{ex:lb_non_interpolation} indicate that the lower bounds produced by SCM are asymptotically less accurate than the upper bounds. Indeed,
this has been observed numerically~\cite{Huynh2010},
implying a need to find more accurate lower bounds. However, the following theorem
indicates that such an improvement is not possible without taking additional information on $A(\mu)$ into account.

\begin{theorem}
	\label{thm:lower_bound_limitations}
	Let $C_J=\{\mu_1,\dots,\mu_J\}\subseteq D$ and consider the lower bounds 
	$\llb(\mu;C_J)$ defined in~\eqref{eq:minimization_bound}
	for a Hermitian matrix function $A(\mu)$ in affine linear decomposition~\eqref{eq:affinelinear}. Let $\tilde \mu \in D$. If 
		$J<\calN$
	then there exist matrices $\bar A_1,\dots,\bar A_Q\in \C^{\calN\times\calN}$, defining $\bar A(\mu)=\theta_1(\mu)\bar A_1+\dots+\theta_Q(\mu)\bar A_Q$,
	such that
	\begin{equation}\label{eq:requiredbara} \lambda_{\mathrm{min}}(\bar A(\tilde \mu)) = \llb(\tilde \mu;C_J) 
	 \qquad \text{and}  \qquad 
	 \lambda_{\mathrm{min}}(\bar A(\mu_i)) = \lambda_{\mathrm{min}}(A(\mu_i)) 
	\end{equation}
	hold for $i = 1,\ldots, J$.

\end{theorem}
\begin{proof}
Let the columns of $V \in \C^{\calN \times J}$
and $V_\perp \in \C^{\calN \times (\calN- J)}$
form orthonormal bases of 
$\mathcal{V}=\spn\{v_{1},\dots,v_{J}\}$ and $\mathcal{V}^\perp$,
respectively, where each
$v_{i}$ denotes an eigenvector associated with $\lambda_{\mathrm{min}}(A(\mu_i))$.
Moreover, let $y_{\tilde\mu} \in \ycallb(C_J) \subset \R^Q$
denote a minimizer of~\eqref{eq:minimization_bound} for $\tilde \mu$, that is,
$\llb(\tilde \mu;C_J) = \theta(\tilde \mu)^T y_{\tilde\mu}$.
The rest of the proof consists of showing that the 
matrices defined by
	\begin{equation*}
		\bar A_q := VV^* A_q VV^* + y_{\tilde\mu,q} V_\perp V_\perp^*, \qquad q \in \{1,\dots, Q\},
	\end{equation*}
satisfy~\eqref{eq:requiredbara}. 

Given $u \in \C^\calN$, we can write $u = u_{\calV} + u_\perp$ with
$u_{\calV} \in \calV$ and 
$u_{\perp} \in \calV^\perp$. For any $\mu \in D$,  we therefore have
\begin{equation} \label{eq:relationblabla}
  u^* \bar A(\mu) u = \sum_{q = 1}^Q \theta_q(\mu) \big( u_\calV^* A_q u_\calV +
 y_{\tilde\mu,q} \|u_{\perp}\|_2^2  
 \big)
 = u_\calV^* A(\mu) u_\calV + \theta(\mu)^T y_{\tilde\mu}\, \|u_{\perp}\|_2^2.
\end{equation}
For $\mu = \mu_i$, this yields
\[
 u^* \bar A(\mu_i) u \ge 
 \lambda_{\min}(A(\mu_i)) \|u_\calV\|_2^2 + \lambda_{\min}(A(\mu_i))
 \|u_{\perp}\|_2^2 = \lambda_{\min}(A(\mu_i)) \|u\|_2^2,
\]
where we used that $y_{\tilde\mu} \in \ycallb(C_J)$ implies 
$\theta(\mu_i)^T y_{\tilde\mu}\ge \lambda_{\min}(A(\mu_i))$.
Since equality is attained for $u = v_i$, this implies the second 
equality in~\eqref{eq:requiredbara}.

To show the first equality, we first notice that 
the definition of $\llb(\tilde \mu;C_J)$ implies 
\[u_\calV^* A(\tilde \mu) u_\calV = \theta(\tilde \mu)^T R(u_\calV) \|u_\calV\|_2^2 \ge \llb(\tilde \mu;C_J) \|u_\calV\|_2^2. \]
Inserted into~\eqref{eq:relationblabla} for $\mu = \tilde \mu$, this 
yields
\[
 u^* \bar A(\tilde \mu) u \ge \llb(\tilde \mu;C_J) \|u_\calV\|_2^2 + \llb(\tilde \mu;C_J) \|u_\perp\|_2^2 = \llb(\tilde \mu;C_J) \|u\|_2^2.
\]
Since equality is attained by any $u \in \calV^\perp$, this 
shows the first equality in~\eqref{eq:requiredbara} and thus
completes the proof.
\end{proof}

Since the definition of the lower bounds in~\eqref{eq:minimization_bound} only depends
on $\theta(\mu)$ and the eigenvalues at $\mu_i$, the lower bounds for the matrix function $\bar A(\mu)$ constructed in Theorem~\ref{thm:lower_bound_limitations} are identical with those for $A(\mu)$. For $\bar A(\mu)$, the lower bound $\llb(\tilde\mu;C_J)$ coincides with the exact eigenvalue at an arbitrary fixed $\tilde \mu \in D$. Hence, additional knowledge, beyond the eigenvalues at $\mu_i$, needs to be incorporated to improve the lower bounds.
%

\section{Subspace approach} \label{sec:subspace}

In this section, our new subspace approach is presented that takes eigenvector information across different parameter samples into account and offers the flexibility to incorporate eigenvectors for larger eigenvalues as well.


Given $C_J=\{\mu_1,\dots,\mu_J\} \subset D$, suppose that for each sample $\mu_i$ we have  computed the $\ell \geq 1$ smallest eigenvalues 
\[\lambda_i = \lambda_i^{(1)} \le  \lambda_i^{(2)} \le \cdots \le \lambda_i^{(\ell)}\]  of $A(\mu_i)$
along with an orthonormal basis of associated eigenvectors 
 $v_i^{(1)}, v_i^{(2)}, \ldots, v_i^{(\ell)} \in \C^\calN$. 
To simplify notation, we assume that an equal number of eigenpairs has been computed
for each $\mu_1,\dots,\mu_J$, although this is not necessary. 
The eigenvectors will be collected in the subspace 
\begin{equation}\label{eq:vectorsubspace}
	\calV(C_J,\ell) := \spn\{v_1^{(1)},\dots,v_1^{(\ell)},v_2^{(1)},\dots,v_2^{(\ell)},\dots,v_J^{(1)},\dots,v_J^{(\ell)}\}.
\end{equation}
In the subsequent two sections, we discuss how the information in
$\calV(C_J,\ell)$ can be used to compute tighter bounds for $\lminmu$.

\subsection{Subspace approach for upper bounds} \label{sec:subspace_ub}

Given the subspace $\calV(C_J,\ell)$ from~\eqref{eq:vectorsubspace},
we define an upper bound set analogously to~\eqref{eq:upper_set}:
\[
	\ycalsub(C_J,\ell) := \{R(v):  v\in \calV(C_J,\ell) \}.
\]
The corresponding upper bound for $\mu \in D$ is defined as
\[
	\lsub(\mu;C_J,\ell) := \min_{y\in \ycalv(C_J,\ell)} \theta(\mu)^Ty.
\]
Clearly, we have $\ycalub(C_J) \subseteq \ycalsub(C_J,\ell) \subseteq \ycal$
and thus \[\lub(\mu;C_J) \ge \lsub(\mu;C_J,\ell) \ge \lminmu.\]

To evaluate $\lsub(\mu;C_J,\ell)$, we first compute an orthonormal basis 
$V \in \C^{\calN\times J\ell}$ of $\calV(C_J,\ell)$
and obtain
\begin{eqnarray}
	\lsub(\mu;C_J,\ell) &=& \min_{v \in \calV(C_J,\ell) } \theta(\mu)^TR(v) 
	= \min_{w\in \C^{J\ell} \atop \|w\|_2 = 1} \theta(\mu)^TR(Vw) \nonumber \\
	&=& \min_{w\in \C^{J\ell} \atop \|w\|_2 = 1} \theta_1(\mu) w^*V^*A_1Vw + \cdots + \theta_Q(\mu) w^*V^*A_QVw \nonumber \\
	&=& \lambda_{\mathrm{min}}\big(\theta_1(\mu)V^*A_1V + \dots \theta_Q(\mu)V^*A_QV\big)=\lambda_{\mathrm{min}}(V^*A(\mu)V). \label{eq:smallevp}
\end{eqnarray}
Thus, the computation of $\lsub(\mu,C_J,\ell)$ requires the solution 
of an eigenvalue problem of size $J\ell \times J\ell$, with 
$J\ell$ usually much smaller than $\calN$.

\subsection{Subspace approach for lower bounds} \label{sec:subspace_lb}

We will use a perturbation result to turn the upper bound~\eqref{eq:smallevp}
into a lower bound $\lslb(\mu;C_J,\ell)$ for $\mu \in D$. For this purpose, we
consider for some small integer $r\le J\ell$ the $r$ smallest eigenvalues
\[
\lsub(\mu;C_J,\ell) = \lambda^{(1)}_{\calV} \le \lambda^{(2)}_{\calV} \le \cdots \le \lambda^{(r)}_{\calV}
\]
of $V^*A(\mu)V$, along with the corresponding eigenvectors $w_1,\ldots, w_r \in \C^{J\ell}$. Let $U \in \C^{\calN\times r}$ be an orthonormal basis of the subspace
$\calU(\mu)$ spanned by the Ritz vectors:
\[\calU(\mu) := \spn\{ V w_1,\ldots, V w_r\}.\] Moreover,
let $U_\perp \in \C^{\calN\times (\calN-r)}$ be an orthonormal basis of
$\calU^\perp(\mu)$ and denote the eigenvalues of $U_\perp^*A(\mu) U_\perp$ by
\[
\lambda^{(1)}_{\calU^\perp} \le \lambda^{(2)}_{\calU^\perp} \le \cdots \le \lambda^{(\calN-r)}_{\calU^\perp}.
\]
The transformed matrix
\[
[U,U_\perp]^* A(\mu) [U,U_\perp] = 
\begin{bmatrix}U^*A(\mu) U & U^*A(\mu) U_\perp \\ U_\perp^*A(\mu) U & U_\perp^*A(\mu) U_\perp \end{bmatrix}
\]
clearly has the same eigenvalues as $A(\mu)$, while the perturbed matrix
\[
 \begin{bmatrix}U^*A(\mu) U & 0 \\ 0 & U_\perp^*A(\mu) U_\perp \end{bmatrix}
\]
has the eigenvalues $\big\{\lambda^{(1)}_{\calV},\ldots,\lambda^{(r)}_{\calV}\big\} \cup \big\{ \lambda^{(1)}_{\calU^\perp},\ldots, \lambda^{(\calN-r)}_{\calU^\perp} \big\}$.
Applying a perturbation result by Li and Li~\cite{Li2005} to this situation yields
the error bound
\[
 \big|\lminmu - \min\big(\lambda^{(1)}_{\calV},\lambda^{(1)}_{\calU^\perp} \big)\big| \leq \frac{2\rho^2}{\delta + \sqrt{\delta^2 + 4\rho^2}},
\]
with the residual norm \[
\rho := \| U_\perp^*A(\mu) U \|_2
= \|A(\mu) U - U (U^* A(\mu) U)\|_2
\] and the absolute gap
$\delta := |\lambda^{(1)}_{\calV} - \lambda^{(1)}_{\calU^\perp}|$.
Rearranging terms thus gives the lower bound
\begin{equation} \label{eq:theoreticallowerbound}
 f(\lambda^{(1)}_{\calU^\perp}) \le \lminmu, \quad \text{with} \quad f(\eta):= \min\big(\lambda^{(1)}_{\calV},\eta \big) - \frac{2\rho^2}{|\lambda^{(1)}_{\calV} - \eta| + \sqrt{|\lambda^{(1)}_{\calV} - \eta|^2 + 4\rho^2}}.
\end{equation}
This lower bound is not practical so far, as it involves the quantity $\lambda^{(1)}_{\calU^\perp}$, which would require the solution of a large eigenvalue problem of size $(\calN - r)\times (\calN - r)$.
\begin{lemma}\label{lemma:function_analysis}
	The function $f:\R\rightarrow\R$ defined in~\eqref{eq:theoreticallowerbound}
	is continuous and monotonically increasing.
\end{lemma}
\begin{proof}
See Section~\ref{sec:proofs}.
\end{proof}
Lemma~\ref{lemma:function_analysis} implies that $f(\eta)$
remains a lower bound as long as $\eta \le \lambda^{(1)}_{\calU^\perp}$. To summarize, our subspace-accelerated lower bound is defined as 
\begin{equation}\label{eq:subspace_lower_bound}
\lslb(\mu;C_J,\ell):=\min\big(\lambda^{(1)}_{\calV},\eta \big) - \frac{2\rho^2}{|\lambda^{(1)}_{\calV} - \eta| + \sqrt{|\lambda^{(1)}_{\calV} - \eta|^2 + 4\rho^2}}
\end{equation}
for a lower bound $\eta$ of $\lambda^{(1)}_{\calU^\perp}$.

\subsubsection{Determining a lower bound for $\lambda^{(1)}_{\calU^\perp}$}\label{sec:ritzbound}

The lower bound for $\lambda^{(1)}_{\calU^\perp} = \lambda_{\min}(U_\perp^* A(\mu)U_\perp)$
needed in~\eqref{eq:subspace_lower_bound} will be determined by adapting the ideas from Section~\ref{sec:scm_idea}.
Let us recall that SCM determines a lower bound for $\lambda_{\min}(A(\mu))$ by solving the LP 
\begin{equation}\label{eq:SCM_lb}
	\llb(\mu;C_j) = \min_{y\in \ycallb(C_J)} \theta(\mu)^Ty,
\end{equation}
with $\ycallb(C_J):= \{ y\in \calB : \theta(\mu_i)^Ty \geq \lambda_i, i=1,\dots,J \}$
and the bounding box $\calB$ defined in~\eqref{eq:bounding_box}. To simplify the discussion, we always assume in the following that $\ycallb(C_J)$ is
a simple polytope with no degenerate facets. Then there exists an optimizer $y_\mu \in \R^Q$ of~\eqref{eq:SCM_lb} such that there are $Q$, among $2Q+J$, linearly independent active 
constraints~\cite{Matousek2007}. In other words,
$y_\mu$ satisfies a linear system
\begin{equation} \label{eq:linearsystemactive}
 	\Theta y_{\mu} = \psi,
\end{equation}
where $\Theta \in \R^{Q\times Q}$ is invertible and each equation corresponds either to a constraint of the form $\theta(\mu_i)^Ty_\mu = \lambda_i$ or to a box
constraint. In the following, we tacitly assume that at least one of the active constraints is a non-box constraint. 

Establishing a lower bound for $\lambda^{(1)}_{\calU^\perp}$ is equivalent to
determining $\eta$ such that $\eta \le u_\perp^* A(\mu) u_\perp$
holds for every $u_\perp \in \calU^\perp(\mu)$ with $\|u_\perp\|_2 = 1$. The restriction of $u_\perp$ to a lower-dimensional subspace
can be used to tighten the non-box constraints in~\eqref{eq:SCM_lb}. 

\begin{lemma} \label{lemma:lowerbound}
 With the notation introduced above, let $\Lambda_i = \text{\rm diag}\big(\lambda_i^{(1)},\ldots,\lambda_i^{(\ell)}\big)$ and
 $V_i = \big[v_i^{(1)}, \ldots, v_i^{(\ell)}\big]$. If $\calN - r \geq r$ then
 \[
  u_\perp^*A(\mu_i)u_\perp \ge \lambda_i + \beta_i,
 \]
 where $\beta_i$ is the smallest eigenvalue of the matrix
 \[
  (\Lambda_i - \lambda_i  I_\ell ) - V_i^* U U^* V_i \big(\Lambda_i - \lambda_i^{(\ell + 1)} I_\ell  \big).
 \]
\end{lemma}
\begin{proof}
Using the spectral decomposition of $A(\mu_i)$, the result follows from
\begin{eqnarray}
\min_{u_\perp \in \calU^\perp \atop \|u_\perp\|_2 = 1} u_\perp^*A(\mu_i)u_\perp & \ge & 
\min_{u_\perp \in \calU^\perp \atop \|u_\perp\|_2 = 1} u_\perp^* V_i \Lambda_i V_i^*u_\perp + \lambda_i^{(\ell + 1)}
u_\perp^* (I - V_i V_i^*) u_\perp  \label{eq:firstinequality} \\
&= &\lambda_i^{(\ell + 1)} +
\min_{u_\perp \in \calU^\perp \atop \|u_\perp\|_2 = 1}  u_\perp^* V_i \big(\Lambda_i - \lambda_i^{(\ell + 1)} I_\ell\big) V_i^*u_\perp \nonumber \\
&=& \lambda_i^{(\ell + 1)} + \lambda_{\min}\big(
U_\perp^* V_i \big(\Lambda_i - \lambda_i^{(\ell + 1)} I_\ell\big) V_i^* U_\perp \big) \nonumber \\
&=& \lambda_i^{(\ell + 1)} + \lambda_{\min}\big(
V_i^* U_\perp U_\perp^* V_i \big(\Lambda_i - \lambda_i^{(\ell + 1)} I_\ell\big)  \big) \nonumber \\
&=& \lambda_i^{(\ell + 1)} + \lambda_{\min}\big( (I_\ell - 
V_i^* U U^* V_i) \big(\Lambda_i - \lambda_i^{(\ell + 1)} I_\ell\big)  \big) \nonumber \\
&=& \lambda_i + \lambda_{\min}\big( (\Lambda_i - \lambda_i  I_\ell ) - V_i^* U U^* V_i \big(\Lambda_i - \lambda_i^{(\ell + 1)} I_\ell  \big)  \big), \nonumber 
\end{eqnarray}
where we used in the third equality that the negative eigenvalues of the matrix product
$U_\perp^* V_i \big(\Lambda_i - \lambda_i^{(\ell + 1)} I_\ell\big) V_i^* U_\perp$
do not change under a cyclic permutation of its factors.
\end{proof}

Using the values of $\beta_i$ defined in Lemma~\ref{lemma:lowerbound}, we update the right-hand side $\psi \in \R^Q$
in~\eqref{eq:linearsystemactive} as follows: If the $k$th equation corresponds to a non-box constraint
$\theta(\mu_i)^Ty = \lambda_i$, we set $\tilde \psi_k := \psi_k + \beta_i = \lambda_i + \beta_i$
and, otherwise, $\tilde \psi_k := \psi_k$.
Since $\Theta$ is invertible, the solution of the resulting LP
\[
	\inf_y \theta(\mu)^Ty \quad \text{subject to} \quad \Theta y \geq \tilde \psi
\]
is trivially given by
\begin{equation} \label{eq:checky}
	\widecheck y_\mu :=\Theta^{-1} \tilde \psi.
\end{equation}
This finally yields the desired lower bound
\[
 \eta(\mu) := \theta(\mu)^T \widecheck y_\mu \leq \lambda^{(1)}_{\calU^\perp} = \lambda_{\min}(U_\perp^* A(\mu) U_\perp).
\]

\begin{remark} \label{remark:choiceofr}
The choice of $r$, the dimension of the Ritz subspace $\calU(\mu)$, requires some consideration.
For $r = 0$, $\calU_\perp(\mu) = \R^\calN$ yields no improvement: $\lslb(\mu;C_J,\ell) = \llb(\mu;C_J,\ell)$.
Intuitively, choosing $r = 1$ will be most effective when the second smallest eigenvalue of $A(\mu)$ is well separated from the smallest eigenvalue. Otherwise, one may benefit from choosing slightly larger values of $r$. In practice, we choose $r$ by taking the maximal value of
$\lslb(\mu;C_J,\ell)$ for a few small values of $r=0,1,2,\dots$.
\end{remark}

\subsection{Algorithm and computational complexity } \label{sec:subspace_complexity}

The implementation of the proposed lower and upper bounds requires some care in order to avoid
unnecessary  that involve quantities of size $\calN$, the original size of the problem.
\begin{description}
\item[Computation of $\rho$.] The quantity
$\rho = \|A(\mu) U - U \Lambda_U\|_2$
with $\Lambda_U = U^* A(\mu) U = \text{diag}\big(\lambda_{\calV}^{(1)},\ldots,\lambda_{\calV}^{(r)}\big)$
can be computed by solving an $r\times r$ eigenvalue problem:
	\begin{eqnarray}
		\rho^2 &=& \lambda_{\max}((A(\mu)U(\mu) - U(\mu)\Lambda(\mu))^*(A(\mu)U(\mu) - U(\mu)\Lambda(\mu))) \nonumber\\
			&=& \lambda_{\max}(U(\mu)^*A(\mu)^*A(\mu)U(\mu) - \Lambda(\mu)^2).
			\nonumber 
	\end{eqnarray}

\item[Computation of $V^* A(\mu) V$ and $U^*A(\mu)^*A(\mu) U$.] By the affine linear decomposition~\eqref{eq:affinelinear},
\[
 V^* A(\mu) V = \theta_1(\mu) V^* A_1 V + \cdots + \theta_Q(\mu) V^* A_Q V.
\]
A standard technique in RBM, we compute and store the $J\ell \times J\ell$ matrices $V^* A_q V$,
and update them as new columns are added to $V$. In turn, the computation of 
$V^* A(\mu) V$, which is needed to evaluate the upper bound for every $\mu \in \Xi$,
becomes negligible as long as $J\ell \ll \calN$.
Similarly, the evaluation of $U^* A(\mu)^*A(\mu) U$ needed for $\rho$
becomes negligible after the  precomputation of 
$V^*A_q^*A_{q^\prime} V$ for all $q,q^\prime=1,\dots,Q$.\\

\item[Choice of $\ell$.] Clearly, a larger choice of $\ell$ can be expected to lead to better bounds. 
On the other hand, a larger value of $\ell$ increases the computational cost. Intuitively, choosing $\ell$ larger than one appears to be most beneficial when
the gap between the smallest and second smallest eigenvalues is small or even vanishes. One could, for example, choose $\ell$ such that
$\lambda_i^{(\ell+1)} - \lambda_i^{(1)}$ exceeds a certain threshold. However, in the absence of a priori information on eigenvalue gaps,
it might be wisest to simply choose $\ell = 1$ for all $\mu_i$.
\end{description}

\begin{algorithm}
\floatname{algorithm}{Algorithm}
\caption{Subspace SCM}
\begin{algorithmic}[1]
\REQUIRE Training set $\Xi$, affine linear decomposition such that  $A(\mu)=\theta_1(\mu)A_1+\dots+\theta_Q(\mu)A_Q$ is
Hermitian for every $\mu \in \Xi$. Relative error tolerance $\varepsilon_{\mathrm{SCM}}$.
\ENSURE Set $C_J \subset \Xi$ with corresponding eigenvalues $\lambda_i^{(j)}$ and eigenvector basis $V$ of
$\calV(C_J,\ell)$, such that $\frac{\lsub(\mu;C_J,\ell) - \lslb(\mu;C_J,\ell)}{\lsub(\mu;C_J,\ell)} < \varepsilon_{\mathrm{SCM}}$ for every $\mu \in \Xi$.
\STATE compute $\lambda_{\min}(A_q), \lambda_{\max}(A_q)$ for $q = 1,\ldots,Q$, defining $\calB$ according to~\eqref{eq:bounding_box}
\STATE $J=0$, $C_0=\emptyset$
\WHILE{ $\max\limits_{\mu \in \Xi} \frac{\lsub(\mu;C_J,\ell) - \lslb(\mu;C_J,\ell)}{\lsub(\mu;C_J,\ell)} > \varepsilon_{\mathrm{SCM}}$  }
  \STATE $\mu_{J+1} \leftarrow \argmax\limits_{\mu \in \Xi} \frac{\lsub(\mu;C_J,\ell) - \lslb(\mu;C_J,\ell)}{\lsub(\mu;C_J,\ell)} $
  \STATE compute smallest eigenpairs $(\lambda_{J+1}^{(1)},v_{J+1}^{(1)}),\dots,(\lambda_{J+1}^{(\ell)},v_{J+1}^{(\ell)})$ of $A(\mu_{J+1})$
  \STATE $C_{J+1} \leftarrow C_J \cup \mu_{J+1}$
  \STATE update $V^*A_qV$ and $V^*A_q^*A_{q^\prime} V$ for all $q,q^\prime =1,\dots,Q$
  \FOR{$\mu \in \Xi$}
	\STATE compute $\lsub(\mu;C_{J+1},\ell) = \lambda_{\mathrm{min}}(V^*A(\mu)V)$
	\STATE compute $\rho = \sqrt{\lambda_{\max}(U(\mu)^*A(\mu)^*A(\mu)U(\mu) - \Lambda(\mu)^2)}$
	\STATE compute $y_\mu = \argmin_{y \in \ycallb(C_{J+1})} \theta(\mu)^Ty$ and updated $\widecheck y_\mu$ according to~\eqref{eq:checky}
	\STATE $\eta(\mu) \leftarrow \theta(\mu)^T \widecheck y_\mu$
	\STATE compute $\lslb(\mu;C_{J+1},\ell)$ according to~\eqref{eq:subspace_lower_bound}
  \ENDFOR
  \STATE $J \leftarrow J+1$
\ENDWHILE
\end{algorithmic}
\label{alg:subspace}
\end{algorithm}

Algorithm~\ref{alg:subspace} summarizes our proposed procedure for computing subspace lower and upper bounds. Similarly as SCM, the algorithm requires 
the solution of $2Q + J$ eigenvalue problems of size
$\calN\times\calN$ for determining the bounding box in the beginning and 
the smallest $\ell+1$ eigenpairs in each iteration. Clearly, the latter
part will become more expensive if $\ell \ge 1$. However, we expect that this increase 
can be mitigated significantly in practice by the use of block algorithms.
More specifically, when using a block eigenvalue solver such as LOBPCG~\cite{Knyazev2001}
and efficient implementations of block matrix-vector products with the matrix $A$ (and its
preconditioner), the computation of $\ell$ smallest eigenvalues will not
be much more expensive as long as $\ell$ remains modest.

Computing $\lsub(\mu;C_J,\ell)$ and $\lslb(\mu;C_J,\ell)$ for all $\mu \in \Xi$ amounts to solving $J|\Xi|$ LP problems with $Q$ variables and $2Q+J$ constraints, as well as $J|\Xi|$ eigenproblems of size $J\ell \times J\ell$. As long as $J\ell \ll \calN$, 
these parts will be negligible, and the cost of Algorithms~\ref{alg:SCM} and~\ref{alg:subspace} will be approximately equal.

\subsection{Interpolation results} \label{sec:subspace_interpolation}

By definition, we already know that the bounds from our subspace approach are never worse than the bounds produced by SCM:
\begin{equation} \label{eq:cascade} 
		\lambda_{\textrm{LB}}(\mu;C_J) \leq\lambda_{\textrm{SLB}}(\mu;C_J,\ell)\leq \lambda_{\min}(A(\mu)) \leq\lambda_{\textrm{SUB}}(\mu;C_J,\ell) \leq \lambda_{\textrm{UB}}(\mu;C_J),
\end{equation}
with equality at $\mu = \mu_i \in C_J$. Together with Theorem~\ref{thm:upper_hermite}, these inequalities imply that our upper bounds also interpolate the derivatives at 
$\mu_i$.

	\begin{corollary}\label{cor:subspace_upper_bound}
	For any $\ell \ge 1$ and any $\mu_i \in C_J$ that satisfies
	the assumptions of Theorem~\ref{thm:upper_hermite},
	it holds that 
		\[
			\nabla \lsub(\mu_i;C_J,\ell) = \nabla \lmin (A(\mu_i))
		\]
with $\lsub(\mu_i;C_J,\ell)$ defined as in~\eqref{eq:smallevp}.
\end{corollary}
	\begin{proof}
	By the assumptions, $\mu_i$ is an interior point of $D$ and~\eqref{eq:cascade} implies that
	the inequality 
  $\lmin(A(\mu)) \leq \lsub(\mu;C_J,\ell) \leq \lub(\mu;C_J)$
holds for all $\mu$ in a neighbourhood of $\mu_i$. Combined with the result $\nabla \lmin(A(\mu_i)) = \nabla \lub(\mu_i;C_J)$ of Theorem~\ref{thm:upper_hermite},
this implies $\nabla \lsub(\mu_i;C_J,\ell) = \nabla \lmin (A(\mu_i))$.
	\end{proof}
	
	In contrast to SCM, it turns out that the subspace lower bounds also interpolate the derivative of $\lminmu$ at $\mu \in C_J$.
	To show this, we need the following lemma.

\begin{lemma}\label{lem:region}
Let $\mu_i \in C_J$
satisfy the assumptions of Theorem~\ref{thm:upper_hermite}.
For any $\varepsilon > 0$, there is a neighbourhood $\Omega \subseteq D$ around $\mu_i$ such that 
\begin{eqnarray}
	\big|\lambda_i - \lambda_{\calV}^{(1)}(\mu)| &\leq& \varepsilon \label{eq:st_region_1},\\
	\lambda_i^{(2)} - \eta(\mu)  &\leq&  \varepsilon \label{eq:st_region_2},
\end{eqnarray}
hold for all $\mu \in \Omega$, where $\lambda_{\calV}^{(1)}(\cdot)$ and $\eta(\cdot)$ are defined as in Section~\ref{sec:subspace_lb}.
\end{lemma}
\begin{proof}
By construction, $\lambda_{\calV}^{(1)}(\mu_i) = \lambda_i$ and thus the continuity of the smallest eigenvalue implies that~\eqref{eq:st_region_1}
holds for all $\mu$ in some neighbourhood $\Omega_1$ around $\mu_i$.
It remains to prove~\eqref{eq:st_region_2}.

In the LP~\eqref{eq:minimization_bound} for determining $\lambda_{\textrm{LB}}(\mu_i;C_J)$, which is trivially given by $\lambda_i$, the constraint 
$\theta(\mu_i)^Ty = \lambda_i$ is active. 
Since we assumed that $\ycallb(C_J)$ is a simple polytope with no degenerate facets,
the continuity of $\theta(\mu)$ implies that this constraint remains active in a neighbourhood 
$\Omega_2$: $\theta(\mu_i)^Ty_\mu = \lambda_i$ for all $\mu \in \Omega_2$, where
$y_\mu$ is a minimizer of~\eqref{eq:minimization_bound} for determining $\llb(\mu;C_J)$.
	
	By~\eqref{eq:firstinequality}, the value of $\beta_i$ defined in Lemma~\ref{lemma:lowerbound}
	satisfies
	\[
	\beta_i =
	 \min_{u_\perp \in \calU^\perp(\mu) \atop \|u_\perp \|_2 = 1} u_\perp^* V_i \Lambda_i V_i^* u_\perp + \lambda_i^{(\ell + 1)}
u_\perp^* (I - V_i V_i^*) u_\perp - \lambda_i.\]
The eigenvector $v_i^{(1)}$ belonging to the eigenvalue $\lambda_i = \lambda_{\min}(A(\mu_i))$
is contained in $\calU$ for $\mu = \mu_i$. Once again the simplicity of $\lambda_i$ implies
that the angle between $v_i^{(1)}$ and $\calU$ becomes arbitrarily small as $\mu$ approaches $\mu_i$.
Therefore, for any $\varepsilon > 0$, there is a neighbourhood $\Omega_3$ such that
\[
 \beta_i \ge  \min_{u_\perp \perp v_i^{(1)} \atop \|u_\perp\|_2 = 1} u_\perp^* V_i \Lambda_i V_i^*u_\perp + \lambda_i^{(\ell + 1)}
u_\perp^* (I - V_i V_i^*) u_\perp - \lambda_i - \frac{\varepsilon}{2} = \lambda_i^{(2)} - \lambda_i - \frac{\varepsilon}{2}.
\]
In summary, the vector $\widecheck y_{\mu}$ defined in~\eqref{eq:checky}
satisfies
\begin{equation} \label{eq:bound398}
  \theta(\mu_i)^T \widecheck y_{\mu} = \lambda_i + \beta_i
  \ge \lambda_i + \lambda_i^{(2)} - \lambda_i - \frac{\varepsilon}{2} = \lambda_i^{(2)} - \frac{\varepsilon}{2}. 
\end{equation}

By the invertibility of $\Theta$, the vector $\widecheck y_{\mu}$ remains bounded in
the vicinity of $\mu_i$. Together with the continuity of $\theta(\mu)$, this implies that there is 
a neighbourhood $\Omega_4$ such that 
\[
	|(\theta(\mu) -\theta(\mu_i))^T \widecheck y_{\mu} | \le \frac{\varepsilon}{2}, \quad \forall \mu \in \Omega_4.
\]
Combined with~\eqref{eq:bound398}, this yields
\[
 \eta(\mu) = \theta(\mu)^T \widecheck y_{\mu} \ge \lambda_i^{(2)} - \varepsilon,
\]
which establishes~\eqref{eq:st_region_2}. Setting $\Omega = \Omega_1\cap\Omega_2\cap\Omega_3\cap\Omega_4$ completes the proof.
%
\end{proof}

The following theorem establishes the Hermite interpolation property of the subspace lower bounds.
\begin{theorem}\label{thm:subspace_lower_hermite}
Let $\mu_i \in C_J$ satisfy the assumptions of Theorem~\ref{thm:upper_hermite}
and, additionally, suppose that $r \le \ell$ and $\lambda_i^{(r+1)} > \lambda_i^{(r)}$. Then
\[
			\nabla \lslb(\mu_i;C_J,\ell) = \nabla \lambda_{\min}(A(\mu_i)).
\]	
\end{theorem}
\begin{proof}
By Lemma~\ref{lem:region} and the simplicity of $\lambda_{\min}(A(\mu_i))$, there is $\delta_0 > 0$ such that 
$\eta(\mu) \ge \lambda_{\calV}^{(1)}(\mu) + \delta_0$ for $\mu$ sufficiently close to $\mu_i$. Hence, 
the subspace lower bound~\eqref{eq:subspace_lower_bound} is given by
\begin{equation} \label{eq:subspace_lower_bound_simple}
 \lslb(\mu;C_J,\ell) = \lambda_{\calV}^{(1)}(\mu) - \frac{2\rho^2}{\delta + \sqrt{\delta^2 + 4\rho^2}}, 
\end{equation}
with $\rho = \|U^\perp A(\mu) U\|_2$ and $\delta = |\lambda_{\calV}^{(1)}(\mu) - \eta(\mu)| \ge \delta_0$. 
By Corollary~\ref{cor:subspace_upper_bound}, 
we have
\[
\nabla \lambda_{\calV}^{(1)}(\mu_i) = \nabla \lsub(\mu_i;C_J,\ell) = \nabla \lmin(A(\mu_i)).
\]
Since $\delta$ is bounded from below, the result follows from~\eqref{eq:subspace_lower_bound_simple} if the
gradient of $\rho^2$ is zero. The assumptions $\lambda_i^{(r+1)} > \lambda_i^{(r)}$ and $r\le \ell$ imply that the
invariant subspace belonging to the $r$ smallest eigenvalues of $A(\mu_i)$ is simple and contained in $V$. Hence, standard perturbation results for invariant subspaces~\cite{Stewart1990} yield $\rho = O(\|\mu - \mu_i\|_2)$
for $\mu\to \mu_i$ and therefore $\nabla \rho^2 = 0$, which completes the proof. 
\end{proof}

\subsubsection{A priori convergence in the one-parameter case}

In the following, we analyse the convergence of our subspace bounds
for a special case: We assume that $A(\mu)$
depends analytically on one parameter $\mu \in [-1,1]$ and, moreover,
the eigenvalue $\lambda_{\min}(A(\mu))$ is simple for all $\mu \in [-1,1]$.

Let $E_{R}$ denote the open elliptic disc with foci $\pm 1$ and
the sum of its half axes equal to $R$.
Under the above assumptions, there is $R_0 > 1$ such that the (suitably normalized)
eigenvector $v(\mu)$ belonging to $\lambda_{\min}(A(\mu))$ admits
an analytic extension $v: E_{R_0} \to \C^N$; see, e.g.,~\cite{Kato1995,Reed1978}.
Note that $v$ can be chosen to have norm $1$ on $[-1,1]$, see~\cite[Theorem XII.4]{Reed1978},
but this is not the case on $E_{R_0}$ in general.
Let $C_J=\{\mu_1,\dots,\mu_J\}$ contain the Chebyshev nodes $\mu_i =\cos(\frac{2i-1}{2 J}\pi)$
and set $v_i := v(\mu_i)$. The corresponding vector-valued interpolating
polynomial is given by
\begin{equation} \label{eq:interpoly}
 p_J(\mu) = \ell_1(\mu) v_1 + \cdots + \ell_J(\mu) v_J
\end{equation}
with the Lagrange polynomials $\ell_1,\ldots, \ell_J: [-1,1]\to \R$.
By extending a standard interpolation error result~\cite[Corollary 6.6A]{Mason2003} to vector-valued
functions (see, e.g.,~\cite[Lemma 2.2]{Kressner2011} for a similar extension),
one can show that
\begin{equation} \label{eq:interperror}
  \max_{\mu \in [-1,1]} \| v(\mu) - p_J(\mu) \|_2 \le 
 \frac{(R+R^{-1})M}{(R^{J+1} - R^{-J-1})(R +R^{-1}-2) }
\end{equation}
holds for any $1 < R < R_0$, with $M = \sup\limits_{z \in E_R} \|v(z)\|_2$.
This result is utilized in the proof of the following theorem, which shows exponential
convergence of our subspace bounds.
\begin{theorem} \label{theorem:}
 Under the setting described above, the subspace lower and upper bounds
 for $\ell = r = 1$ satisfy
\begin{eqnarray} \label{eq:upperexpconvergence}
 \lsub(\mu;C_J,1)-\lminmu &\leq& C_U\, R^{-2 J}  \\
 \lminmu - \lslb(\mu;C_J,1) &\leq& C_L\, R^{-2 J} \label{eq:lowerexpconvergence}
\end{eqnarray}
for every $\mu \in [-1,1]$, 
with constants $C_U, C_L$ independent of $J$ 
and $\mu$.
\end{theorem}
\begin{proof}
For $\ell = 1$, the subspace used in our bounds takes the form $\calV = \text{span}\{v_1,\ldots,v_J\}$.
The interpolating polynomial defined in~\eqref{eq:interpoly} clearly satisfies $p_J(\mu) \in \calV$
and hence~\eqref{eq:interperror} yields the following bound on the angle between $\calV$ and $v(\mu)$:
\begin{equation} \label{eq:approxpower}
  \min_{\tilde v \in \calV} \| \tilde v - v(\mu)  \|_2 \lesssim R^{-J}.
\end{equation}
The inequality~\eqref{eq:upperexpconvergence} now follows from existing approximation results for Ritz values; see, e.g.,~\cite[Chapter 11]{Parlett1998}.

To prove~\eqref{eq:lowerexpconvergence}, we first note that the arguments from the proof of Theorem~\ref{thm:subspace_lower_hermite} can be utilized to show that
\[
 \lslb(\mu;C_J,1) = \lsub(\mu;C_J,1) - \frac{2\rho^2}{\delta + \sqrt{\delta^2 + 4\rho^2}}, 
\]
for sufficiently large $J$, where $\delta > \delta_0 > 0$ for some $\delta_0$ 
not depending on $\mu$ or $J$. Since $r = 1$, the quantity $\rho$
coincides with the residual of the smallest Ritz vector of $A(\mu)$ with respect to
$\calV$. By~\eqref{eq:approxpower} and~\cite[Theorem 11.7.1]{Parlett1998}, 
we have $\rho \lesssim R^{-J}$, which completes the proof.
\end{proof}

The maximal value of the exponent $R$ in~\eqref{eq:upperexpconvergence}--\eqref{eq:lowerexpconvergence}
depends on the gap between the smallest and the second smallest eigenvalue and the variation of $A(\mu)$. This is discussed in more detail for a special case in~\cite[Section 2.3.1]{Andreev2012c}.

\subsection{Residual-based lower bounds} \label{sec:heuristic}

As we will see in the numerical experiments (especially in Example~\ref{example:TBaniso}),
the subspace lower bounds can sometimes converge rather slowly in the initial phase of the algorithm,
in contrast to the subspace upper bounds. This slow convergence can be viewed as a price that needs to be
paid in order maintain the reliability of the lower bounds.
In the following, we will propose an alternative that is heuristic (i.e., its
reliability is not guaranteed) and is observed to converge faster in the initial phase.


The alternative consists of simply subtracting the residual norm from the upper bound:
\begin{equation} \label{eq:heuristiclowerbound}
  \lsub(\mu;C_J,\ell) - \|A(\mu) u - \lsub(\mu;C_J,\ell) u\|_2,
\end{equation}
where $u$ with $\|u\|_2 = 1$ is a Ritz vector belonging to the smallest Ritz value $\lsub(\mu;C_J,\ell)$ of $A(\mu)$ with respect to $\calV$.
%
A basic first-order perturbation result for Hermitian matrices~\cite{Parlett1998} implies that~\eqref{eq:heuristiclowerbound} constitutes a lower bound for \emph{an} eigenvalue of $A(\mu)$,
but not necessarily the smallest one. There is a risk, especially in the very beginning, that~\eqref{eq:heuristiclowerbound} is actually larger than the smallest eigenvalue, see Section~\ref{sec:examples} for examples.
However, in all numerical experiments we have observed that a small number of iterations suffices until~\eqref{eq:heuristiclowerbound} becomes a lower bound for the smallest eigenvalue.

\begin{remark}
When using the residual-based lower bound~\eqref{eq:heuristiclowerbound}, it makes sense to also adjust the error measure~\eqref{eq:error_ratio} that drives the sampling strategy to
	\[
		\max_{\mu \in \Xi} \frac{\|A(\mu) u - \lsub(\mu;C_J,\ell) u\|_2}{|\lsub(\mu;C_J,\ell)|},
	\]
	and stop the iteration when this error estimate drops below $\varepsilon_{\mathrm{SCM}}$.
\end{remark}

\section{Applications and numerical experiments} \label{sec:examples}

In this section, we report on the performance of our proposed approach
for a number of large-scale examples. 
Algorithms~\ref{alg:SCM} and~\ref{alg:subspace} have been implemented in {\sc Matlab} Version 7.14.0.739 (R2012a)
and all experiments have been performed on an Intel Xeon CPU E31225 with
4 cores, 3.1 GHz, and 8 GB RAM.

We compare Algorithm~\ref{alg:subspace} with Algorithm~\ref{alg:SCM} by computing the 
maximum relative error ratio~\eqref{eq:error_ratio}.
Additionally, we compare the convergence of the bounds from Sections~\ref{sec:scm} and~\ref{sec:subspace}
towards the exact smallest eigenvalues by measuring the absolute error
\begin{equation} \label{eq:bounderror}
 \max_{\mu \in \Xi} |\text{bound}(\mu) - \lambda_{\min}(A(\mu))|,
\end{equation}
for the corresponding bound,
both with respect to the number of iterations and with respect to the execution time (in seconds).

When implementing and testing Algorithms~\ref{alg:SCM} and~\ref{alg:subspace}, we have made the following choices. We set the relative tolerance to $\varepsilon_{SCM}=10^{-4}$, the 
maximum number of iterations to $J_{\mathrm{max}}=200$ and the surrogate set $\Xi$ to be 
a random subset of $D$ containing 1000 elements. The smallest eigenpairs of $A(\mu_i)$ 
have been computed using the {\sc Matlab} built-in function {\tt eigs}, which is based on ARPACK~\cite{Lehoucq1998}, with the tolerance set to $10^{-6}$.
The {\sc Matlab} built-in function {\tt linprog} with the tolerance set to $10^{-8}$ has been used for solving all linear programs.
In all experiments, we have used Algorithm~\ref{alg:subspace} with the number of smallest eigenpairs included in $\calV$ set to $\ell=1$, 
since this already provided significant improvements over Algorithm~\ref{alg:SCM}. For choosing $r$ from Section~\ref{sec:subspace_lb}, we have tested all values $r=0,1,\dots,Q$, see Remark~\ref{remark:choiceofr}.

\begin{remark}~\label{remark:scm_optimization}
A slight modification of Algorithm~\ref{alg:SCM} can significantly reduce the time spent on solving linear programs. 
For $\mu\in \Xi$, suppose that $y_{\mathrm{LB}}(\mu)\in\ycallb(C_J)$ is 
a minimizer of~\eqref{eq:minimization_bound} on $\ycallb(C_J)$.
Let $(\lambda_{J+1},v_{J+1})$ be the smallest eigenpair of $A(\mu_{J+1})$ computed in iteration $J+1$.
If $\theta(\mu)^T y_{\mathrm{LB}}(\mu) \geq \lambda_{J+1}$, we have
$y_{\mathrm{LB}}(\mu) \in \ycallb(C_{J+1}) \subseteq \ycallb(C_J)$, making $y_{\mathrm{LB}}(\mu)$
also a minimizer of~\eqref{eq:minimization_bound} on $\ycallb(C_{J+1})$. In this case, we have $\llb(\mu;C_{J+1})=\llb(\mu;C_{J})$
and there is no need to solve the linear program in~\eqref{eq:minimization_bound}.
\end{remark}

\subsection{Random matrices}

We first consider an academic example, where a random dense Hermitian matrix
$A_1\in \C^{\calN\times\calN}$ is perturbed, to a certain extent, by random Hermitian matrices $A_2,\dots,A_Q\in \C^{\calN\times\calN}$:
\begin{equation*}
	A(\mu) = A_1 + \mu_2 A_1 + \dots + \mu_Q A_Q,
\end{equation*}
where $\mu = (\mu_2,\dots,\mu_Q) \in D = [0,\delta]^{Q-1}$.

\begin{example}\label{example:random}
	We consider $Q=4$, $\calN = 1000$, $\delta = 0.2$, with $A_1,A_2,A_3,A_4$ having real random entries from the unit normal distribution. 
	The performances of both algorithms is shown in
	Figure~\ref{fig:random}.
	The convergence of Algorithm~\ref{alg:SCM} flattens after around 25 iterations and 
	does not reach the desired tolerance, while the convergence of Algorithm~\ref{alg:subspace} 
	is much faster and reaches the desired tolerance within 47 iterations.
	We have also considered the optimized version of Algorithm~\ref{alg:SCM}, 
	as described in Remark~\ref{remark:scm_optimization}. To make the comparison fair, 
	we have compared it to a variant of Algorithm~\ref{alg:subspace} where the subspace bounds are recomputed only 
	when $\llb(mu;C_J)$ is recomputed. The influence of these modifications on the performances of both algorithms
	can be seen in Figure~\ref{fig:randomopt}. The optimized version of Algorithm~\ref{alg:subspace} requires 
	a slightly larger number of iterations to converge to prescribed accuracy, but it still 
	outperforms even the optimized version of Algorithm~\ref{alg:SCM}.
	Since Algorithm~\ref{alg:subspace} converges quickly, there is no need to even consider the residual-based lower bounds from Section~\ref{sec:heuristic},
	but we still include the results in Figures~\ref{fig:random} and~\ref{fig:randomopt} for the sake of completeness. Here and in the following, the star denotes the iteration from which on the residual-based lower bound~\eqref{eq:heuristiclowerbound} actually constitutes a lower bound for the smallest eigenvalue.
\end{example}

\begin{figure}[ht]
	\resetsubfigs
	\centering
	\begin{subfigure}[t]{0.30\textwidth}\centering
        	\includegraphics[width=\textwidth]{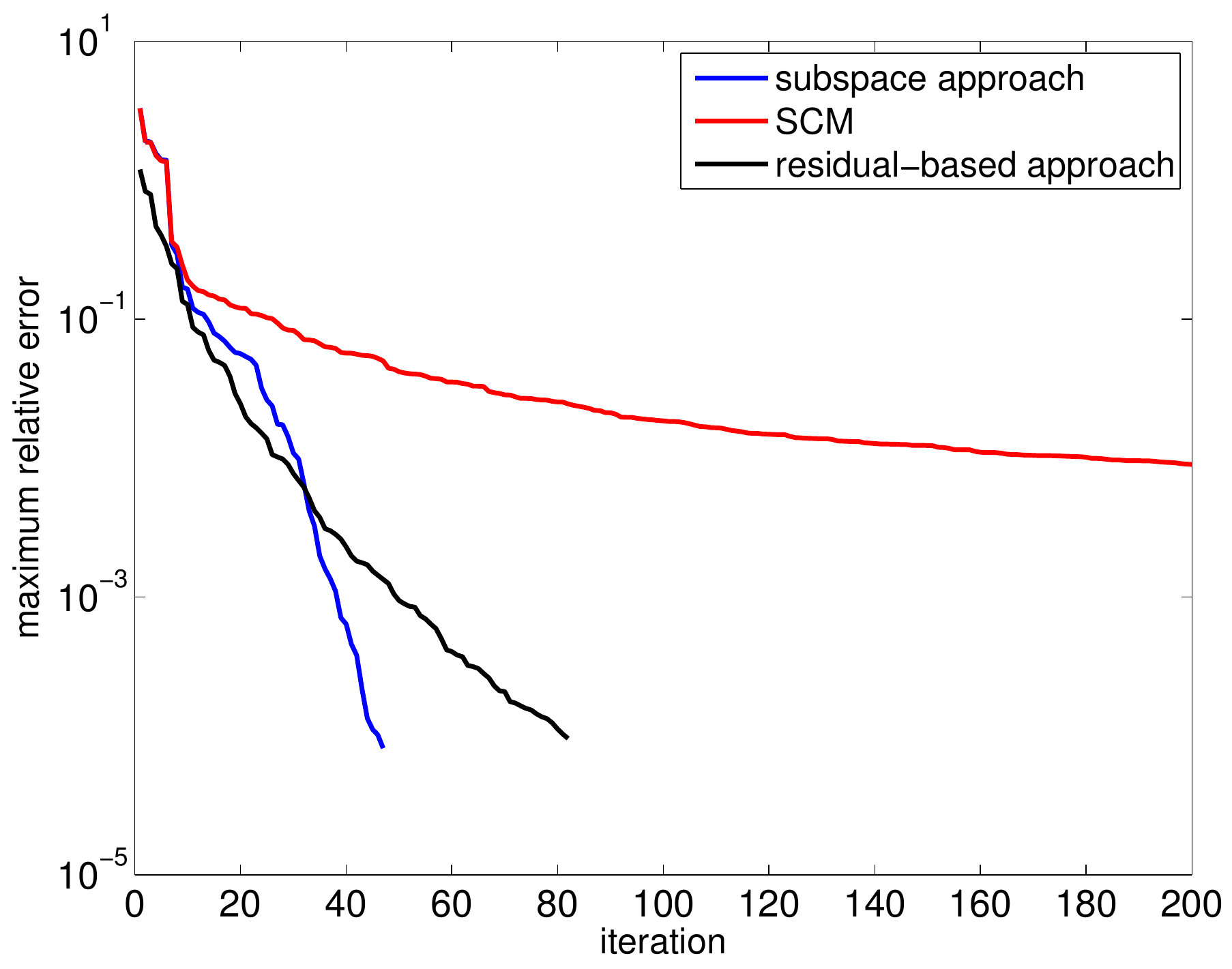}
		\subcaption{Convergence of the maximum relative error ratio~\eqref{eq:error_ratio}.}
		\label{fig:randomerrorratioconv}
	\end{subfigure}
		~
	\begin{subfigure}[t]{0.30\textwidth}
		\includegraphics[width=\textwidth]{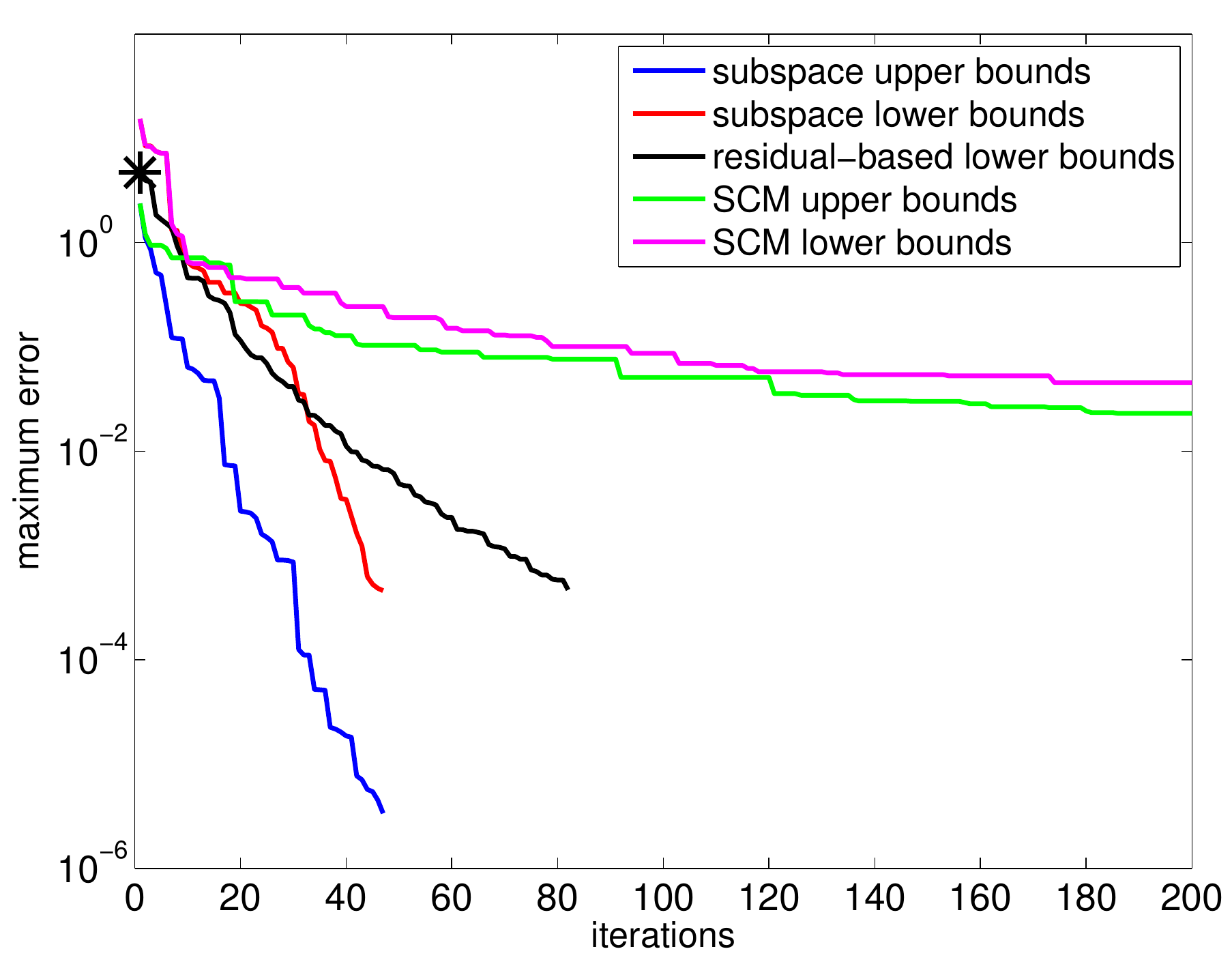}
		\subcaption{Convergence of the error~\eqref{eq:bounderror} for the bounds w.r.t. iteration.}
		\label{fig:randomboundconv}
	\end{subfigure}
	~
	\begin{subfigure}[t]{0.30\textwidth}
		\includegraphics[width=\textwidth]{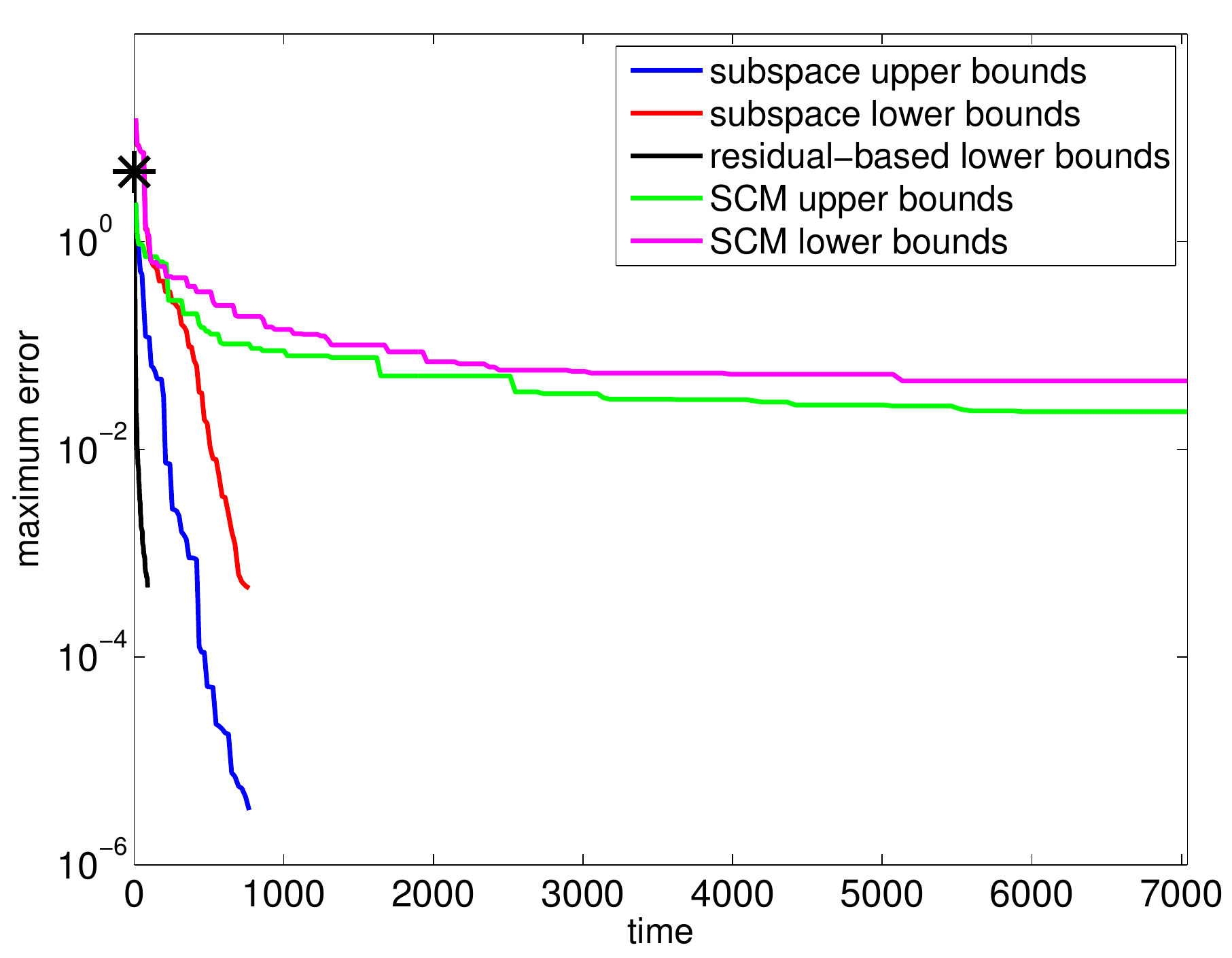}
		\subcaption{Convergence of the error~\eqref{eq:bounderror} for the bounds w.r.t. time.}
		\label{fig:randomtimeerrorratioconv}
	\end{subfigure}
	\caption{Convergence plots for Algorithms~\ref{alg:SCM} and~\ref{alg:subspace} applied to Example~\ref{example:random}.}
	\label{fig:random}
\end{figure}
\begin{figure}[ht]
	\resetsubfigs
	\centering
	\begin{subfigure}[t]{0.30\textwidth}\centering
        	\includegraphics[width=\textwidth]{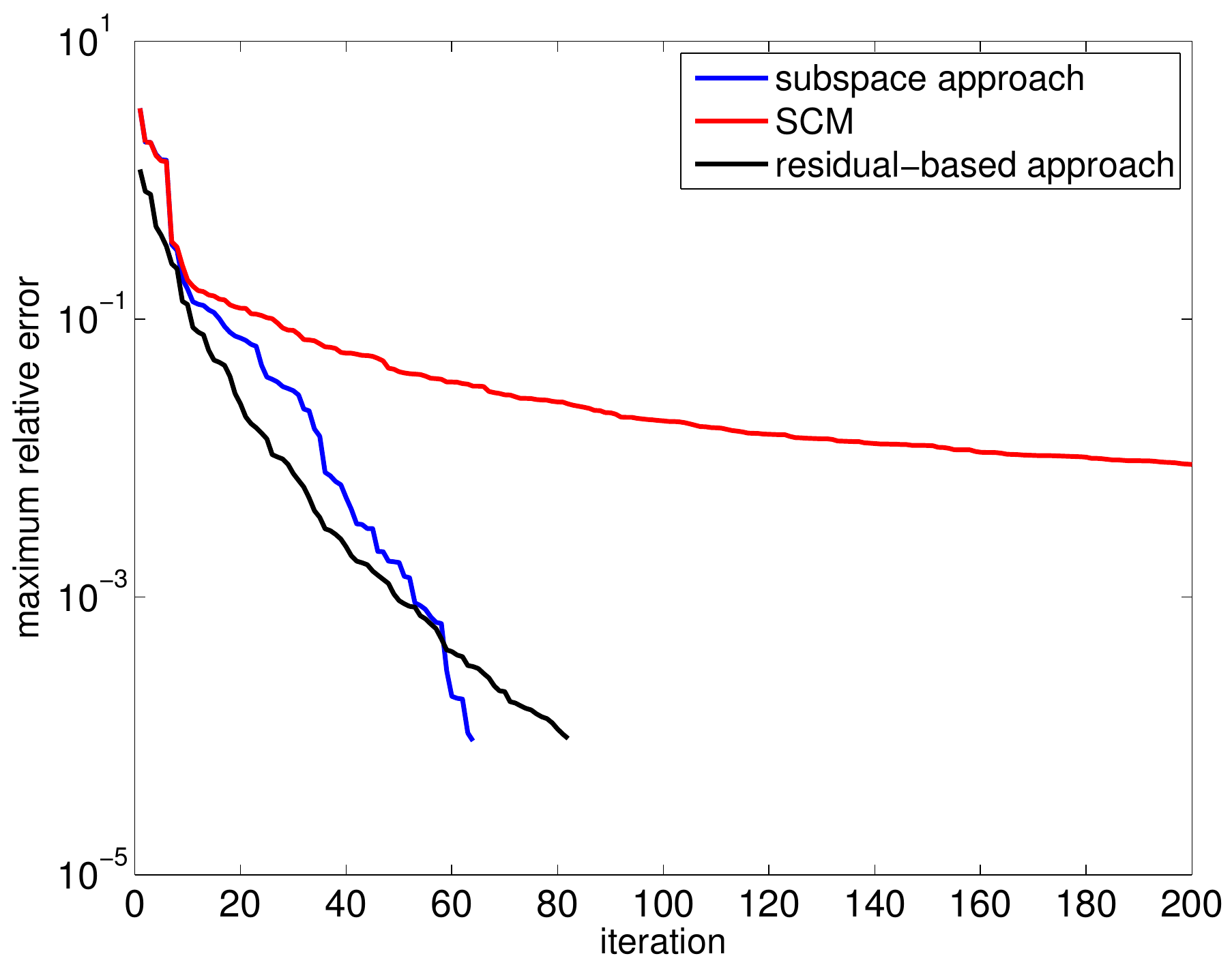}
		\subcaption{Convergence of the maximum relative error ratio~\eqref{eq:error_ratio}.}
		\label{fig:randomopterrorratioconv}
	\end{subfigure}
		~
	\begin{subfigure}[t]{0.30\textwidth}
		\includegraphics[width=\textwidth]{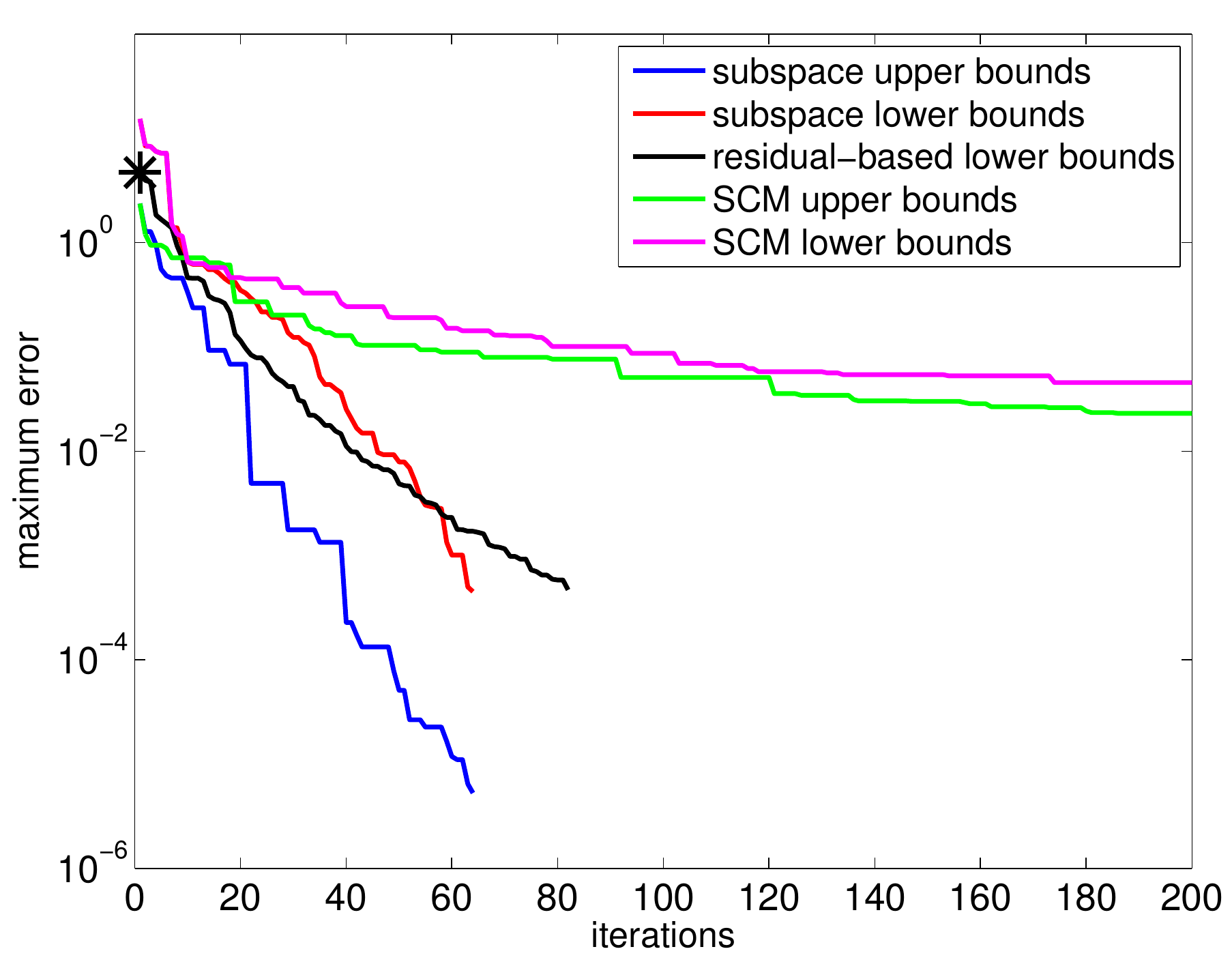}
		\subcaption{Convergence of the error~\eqref{eq:bounderror} for the bounds w.r.t. iteration.}
		\label{fig:randomoptboundconv}
	\end{subfigure}
	~
	\begin{subfigure}[t]{0.30\textwidth}
		\includegraphics[width=\textwidth]{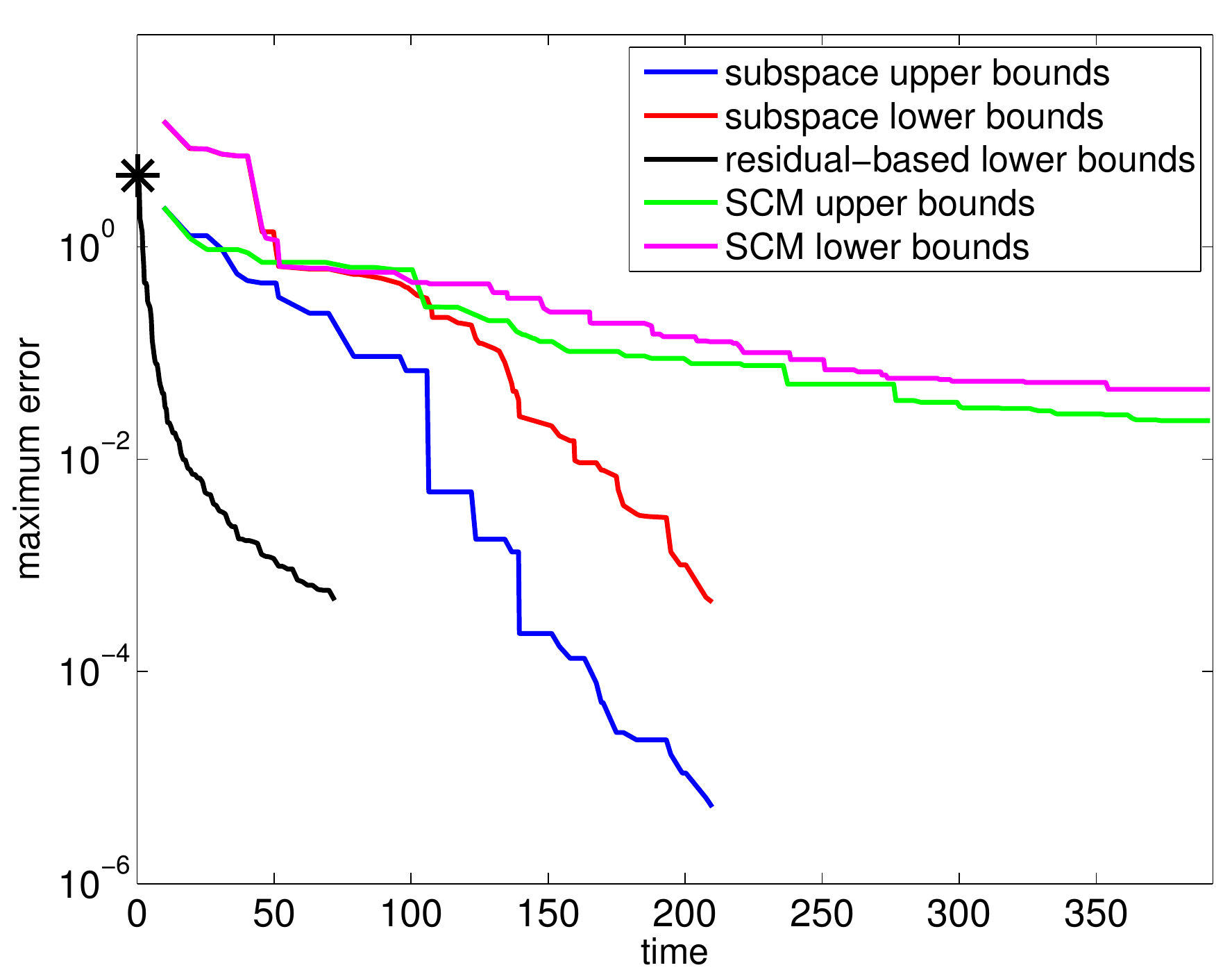}
		\subcaption{Convergence of the error~\eqref{eq:bounderror} for the bounds w.r.t. time.}
		\label{fig:randomopttimeerrorratioconv}
	\end{subfigure}
	\caption{Convergence plots for the optimized versions of Algorithms~\ref{alg:SCM} and~\ref{alg:subspace} 
	applied to Example~\ref{example:random}.}
	\label{fig:randomopt}
\end{figure}

\subsection{Estimation of the coercivity constant} \label{sec:coercivity}

\textit{A posteriori} error estimation in model order reduction techniques for parametrized PDEs, such as reduced basis method, requires reliable estimates for the coercivity constant~\cite{RozH08}
defined as 
\begin{equation}\label{eq:coercivity_constant}
	\alpha(\mu) = \inf_{u \in X} \frac{a(u,u;\mu)}{\|u\|_{X}^2}.
\end{equation}
Here, $a(\cdot,\cdot,\mu)$ is the coercive symmetric bilinear form in the weak formulation of the underlying PDE on the domain $\Omega$ and $X$ 
is the accompanying function space with the norm $\|\cdot\|_X$ induced by the scalar product $(u,v)_X = \tau (u,v)_{L^2(\Omega)} + a(u,v;\bar \mu)$, with $\tau > 0$ and a fixed parameter value $\bar \mu$ chosen to be the center of $D$.
A finite element discretization of~\eqref{eq:coercivity_constant} leads to the minimization problem
\begin{equation} \label{eq:mass_matrix_min}
	\alpha^\calN(\mu) = \inf_{v\in \R^N} \frac{v^TA(\mu)v}{v^TXv}, 
\end{equation}
where $A(\mu), X\in \R^{\calN\times \calN}$ are the matrices discretizing $a(\cdot,\cdot,\mu)$ and  
$(\cdot,\cdot)_X$, respectively.
Minimizing~\eqref{eq:mass_matrix_min} is clearly equivalent to computing the smallest eigenvalue of the generalized eigenvalue problem
	\begin{equation*}
		A(\mu)v = \lambda Xv.
	\end{equation*}
	By computing a (sparse) Cholesky factorization $X=LL^T$, we 
transform~\eqref{eq:par_her_eigp}:
	\begin{equation*}
		L^{-1}A(\mu)L^{-T}w = \lambda w.
	\end{equation*}
	Hence, the matrices $A_i$ appearing in Assumption~\ref{assum:aff_lin_dec}
	need to be replaced by 
	\[
		L^{-1}A_iL^{-T}, \quad i =1,\dots,Q.
	\]
	In the following, we consider three numerical examples of this type from the rbMIT toolbox~\cite{Huynh2010a}. We only include brief explanations of the examples;
	more details can be found in~\cite{Huynh2010a} and~\cite{Patera2012}.

\begin{example}\label{example:square}
	This example 
	concerns a linear elasticity model of a parametrized body (see Figure~\ref{fig:squaregeometry}). 
	The parameter $\mu_1$ determines the width of the hole in the body while the parameter $\mu_2$ determines
	its Poisson's ratio. A discretization of the underlying PDE leads to the matrix $A(\mu)=\sum_{i=1}^Q \theta_i(\mu)A_i$, 
	with $Q=16$, $\mu=(\mu_1,\mu_2)$ and functions $\theta_i(\mu)$ 
	that arise from the parametrization of the geometry. 
	We choose $\calN = 2183$ and $D=[-0.1,0.1]\times[0.2,0.3]$. 
	As can be seen from Figure~\ref{fig:square}, 
	The results are similar to those presented in Example~\ref{example:random}, with
	Algorithm~\ref{alg:subspace} converging in $31$ iteration and Algorithm~\ref{alg:SCM} not reaching the desired tolerance.
\end{example}

\begin{figure}[ht]
	\resetsubfigs
	\centering
	   \begin{subfigure}[t]{0.48\textwidth}\centering
	    \begin{tikzpicture}[scale=2]
	\tikzstyle{dot}=[draw,shape=circle,fill=black,inner sep=0pt,minimum size=3pt]
	\tikzstyle{edge} = [draw,thick,-,black]
	\tikzstyle{dirichlet} = [edge,blue,dashed]
	
	\begin{scope}
		\draw (0,1) node[dot,label={[xshift=0, yshift=0]$(0,1)$}] (p0) {};
		\draw (1,1) node[dot,label={[xshift=0, yshift=0]$(1,1)$}] (p1) {};
		\draw (1,-1) node[dot,label={[xshift=0, yshift=-25]$(1,-1)$}] (p2) {};
		\draw (0,-1) node[dot,label={[xshift=0, yshift=-25]$(0,-1)$}] (p3) {};
		\draw (0,-0.4) node[dot,label={[xshift=-30, yshift=-10]$(0,\mu_1-\frac{1}{2})$}] (p4) {};
		\draw (0.5,-0.4) node[dot,label={[xshift=0, yshift=-25]$(\frac{1}{2},\mu_1-\frac{1}{2})$}] (p5) {};
		\draw (0.5,0.4) node[dot,label={[xshift=0, yshift=0]$(\frac{1}{2},\mu_1+\frac{1}{2})$}] (p6) {};
		\draw (0,0.4) node[dot,label={[xshift=-30, yshift=-10]$(0,\mu_1+\frac{1}{2})$}] (p7) {};
		
		\draw[edge] (p0) -- (p1);
		\draw[edge] (p1) -- (p2);
		\draw[edge] (p2) -- (p3);
		\draw[edge] (p3) -- (p4);
		\draw[edge] (p4) -- (p5);
		\draw[edge] (p5) -- (p6);
		\draw[edge] (p6) -- (p7);
		\draw[edge] (p7) -- (p0);
	\end{scope}

\end{tikzpicture}
		\subcaption{Geometry of the underlying PDE.}
		\label{fig:squaregeometry}
		\end{subfigure}
		~
		\begin{subfigure}[t]{0.48\textwidth}
		\includegraphics[width=\textwidth]{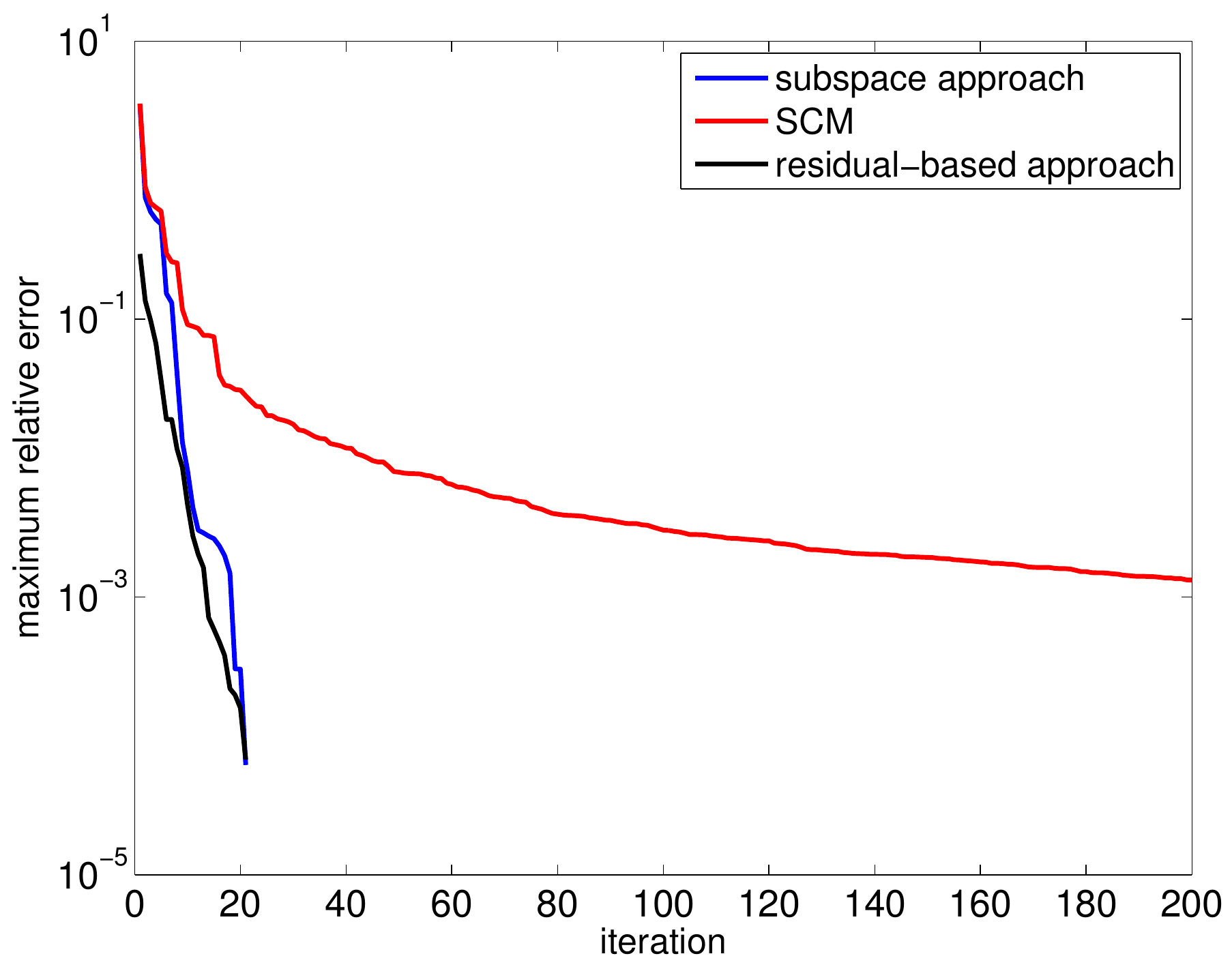}
		\subcaption{Convergence of the maximum relative error ratio~\eqref{eq:error_ratio}.}
		\label{fig:squareerrorratioconv}
		\end{subfigure}
		\\
	 	\begin{subfigure}[t]{0.48\textwidth}
		\includegraphics[width=\textwidth]{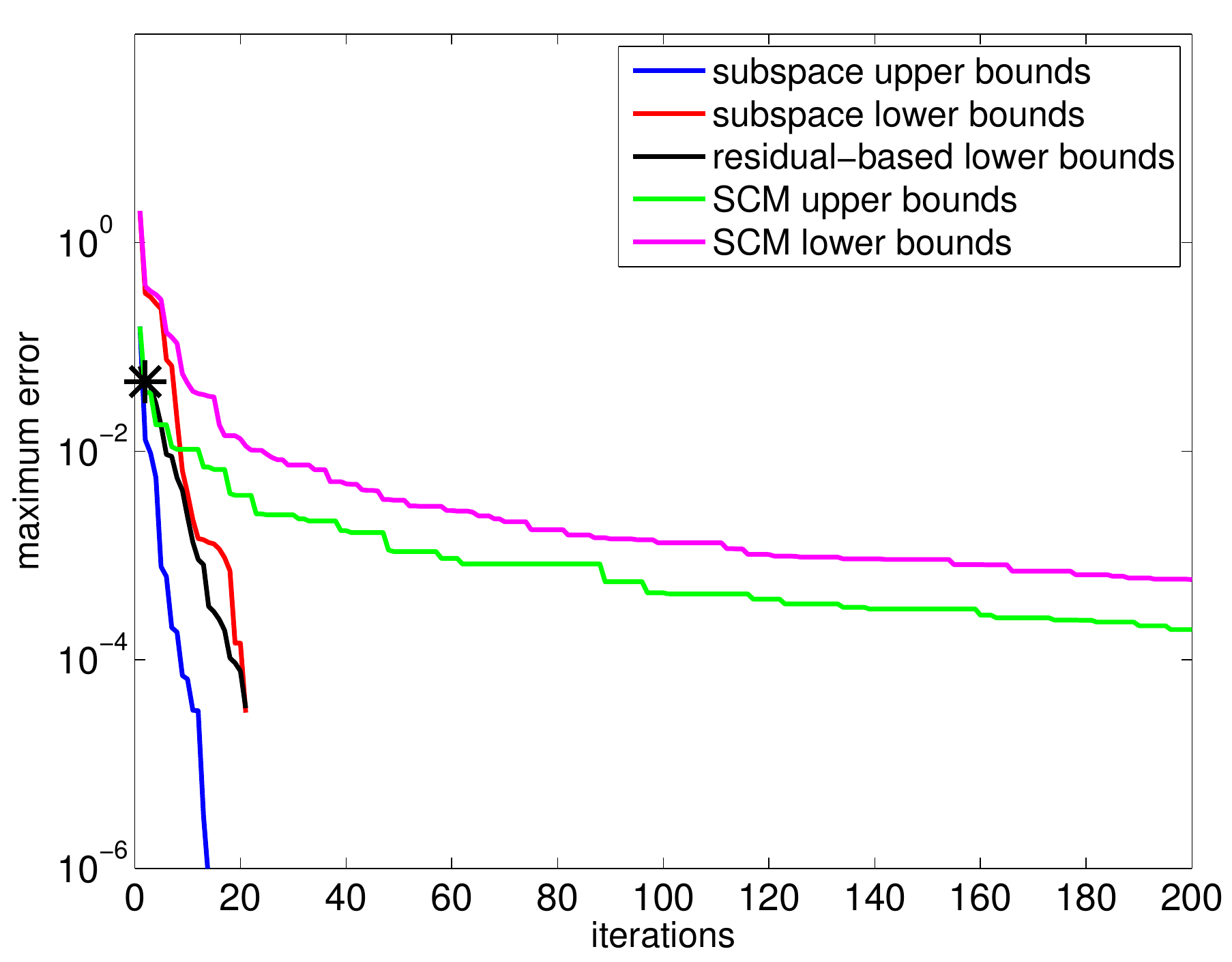}
		\subcaption{Convergence of the error~\eqref{eq:bounderror} for the bounds w.r.t. iteration.}
		\label{fig:squareboundconv}
		\end{subfigure}
		~
		\begin{subfigure}[t]{0.48\textwidth}
		\includegraphics[width=\textwidth]{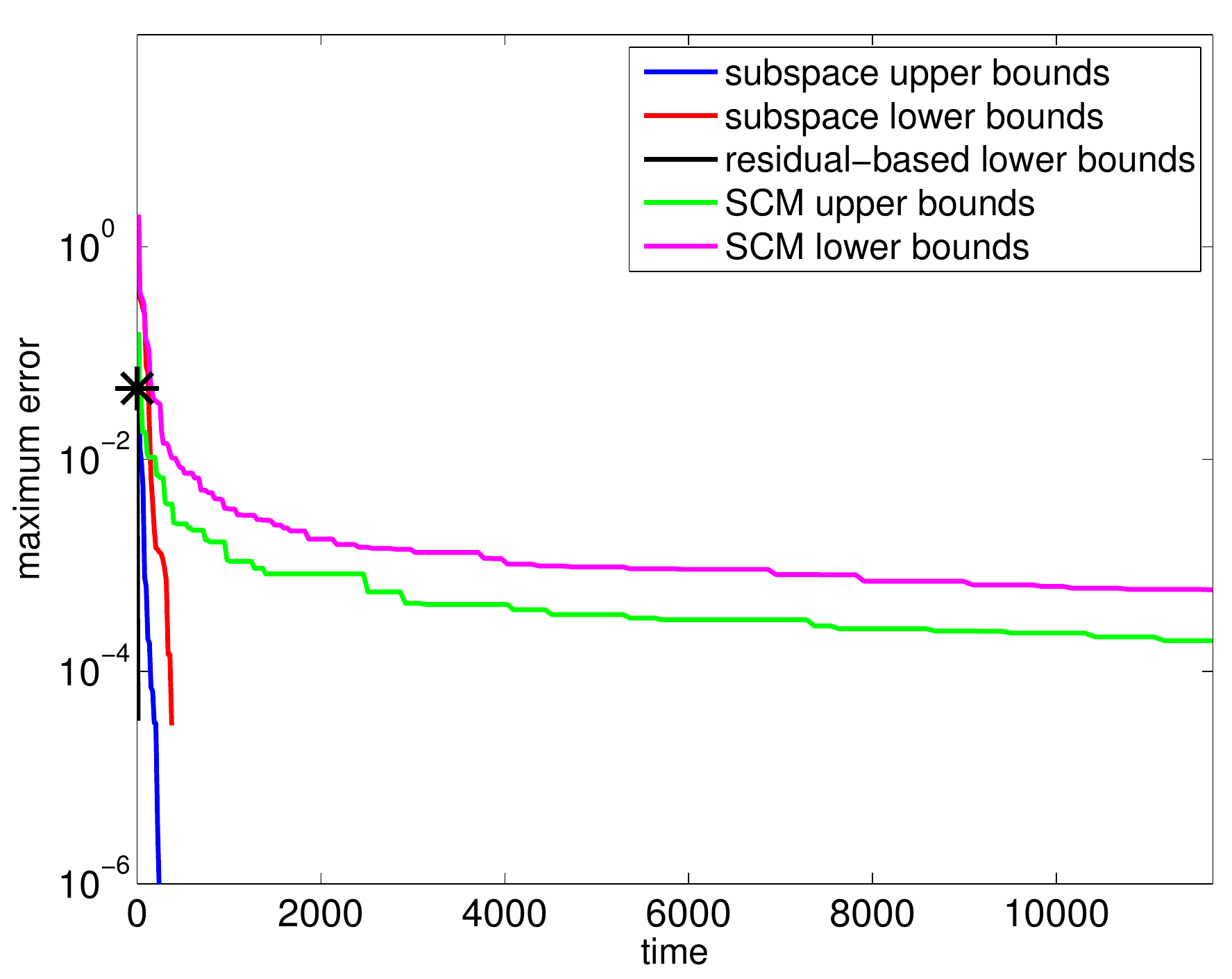}
		\subcaption{Convergence of the error~\eqref{eq:bounderror} for the bounds w.r.t. time.}
		\label{fig:squaretimeboundconv}
		\end{subfigure}
	\caption{Convergence plots for Algorithms~\ref{alg:SCM} and~\ref{alg:subspace} applied to Example~\ref{example:square}.}
	\label{fig:square}
\end{figure}
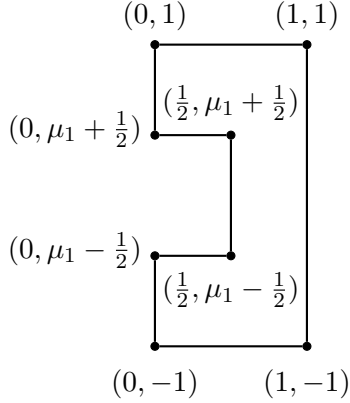
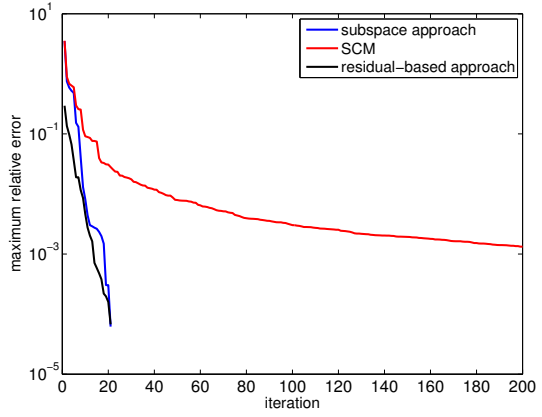
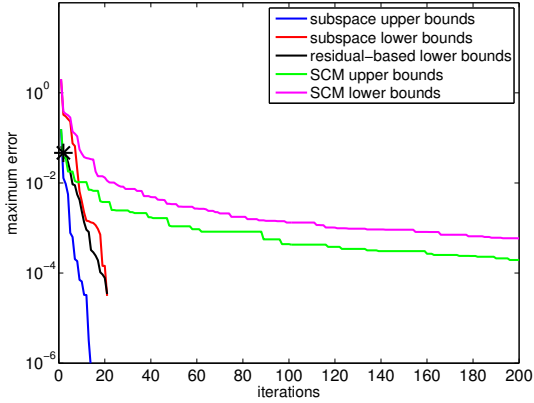
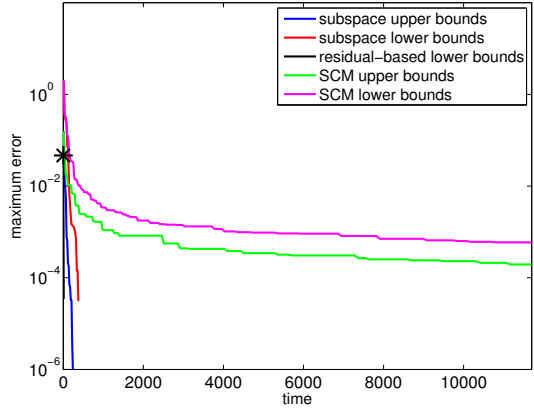

\begin{example}\label{example:fin}
	This example 	concerns a stationary heat equation on a parametrized domain (see Figure~\ref{fig:fingeometry}). 
	The parameter $\mu_1$ determines the coefficient in the Robin boundary conditions while the parameter $\mu_2$ determines
	the length of the domain. A discretization of the underlying PDE leads to the matrix $A(\mu)=\sum_{i=1}^Q \theta_i(\mu)A_i$, 
	with $Q=3$, $\mu=(\mu_1,\mu_2)$ and functions $\theta_i(\mu)$ 
	arising from the parametrization of the geometry and boundary conditions. 
	We choose $\calN = 1311$ and $D=[0.02,0.5]\times[2,8]$. 
As can be seen from Figure~\ref{fig:fin}, the results are similar to those observed in Examples~\ref{example:random} and~\ref{example:square}.
\end{example}

\begin{figure}[ht]
	\resetsubfigs
	\centering
	   \begin{subfigure}[t]{0.48\textwidth}\centering
	    \begin{tikzpicture}[scale=2]
	\tikzstyle{dot}=[draw,shape=circle,fill=black,inner sep=0pt,minimum size=3pt]
	\tikzstyle{edge} = [draw,thick,-,black]
	\tikzstyle{dirichlet} = [edge,blue,dashed]
	
	\begin{scope}
		\draw (-0.5,2) node[dot,label={[xshift=0, yshift=0]$(-\frac{1}{2},\mu_2)$}] (p0) {};
		\draw (0.5,2) node[dot,label={[xshift=0, yshift=0]$(\frac{1}{2},\mu_2)$}] (p1) {};
		\draw (0.5,0) node[dot,label={[xshift=0, yshift=-25]$(\frac{1}{2},0)$}] (p2) {};
		\draw (-0.5,0) node[dot,label={[xshift=0, yshift=-25]$(-\frac{1}{2},0)$}] (p3) {};
		
		\draw[edge] (p0) -- (p1);
		\draw[edge] (p1) -- (p2);
		\draw[edge] (p2) -- (p3);
		\draw[edge] (p3) -- (p0);
	\end{scope}

\end{tikzpicture}
		\subcaption{Geometry of the underlying PDE.}
		\label{fig:fingeometry}
        \end{subfigure}
        ~
        	\begin{subfigure}[t]{0.48\textwidth}
        	\includegraphics[width=\textwidth]{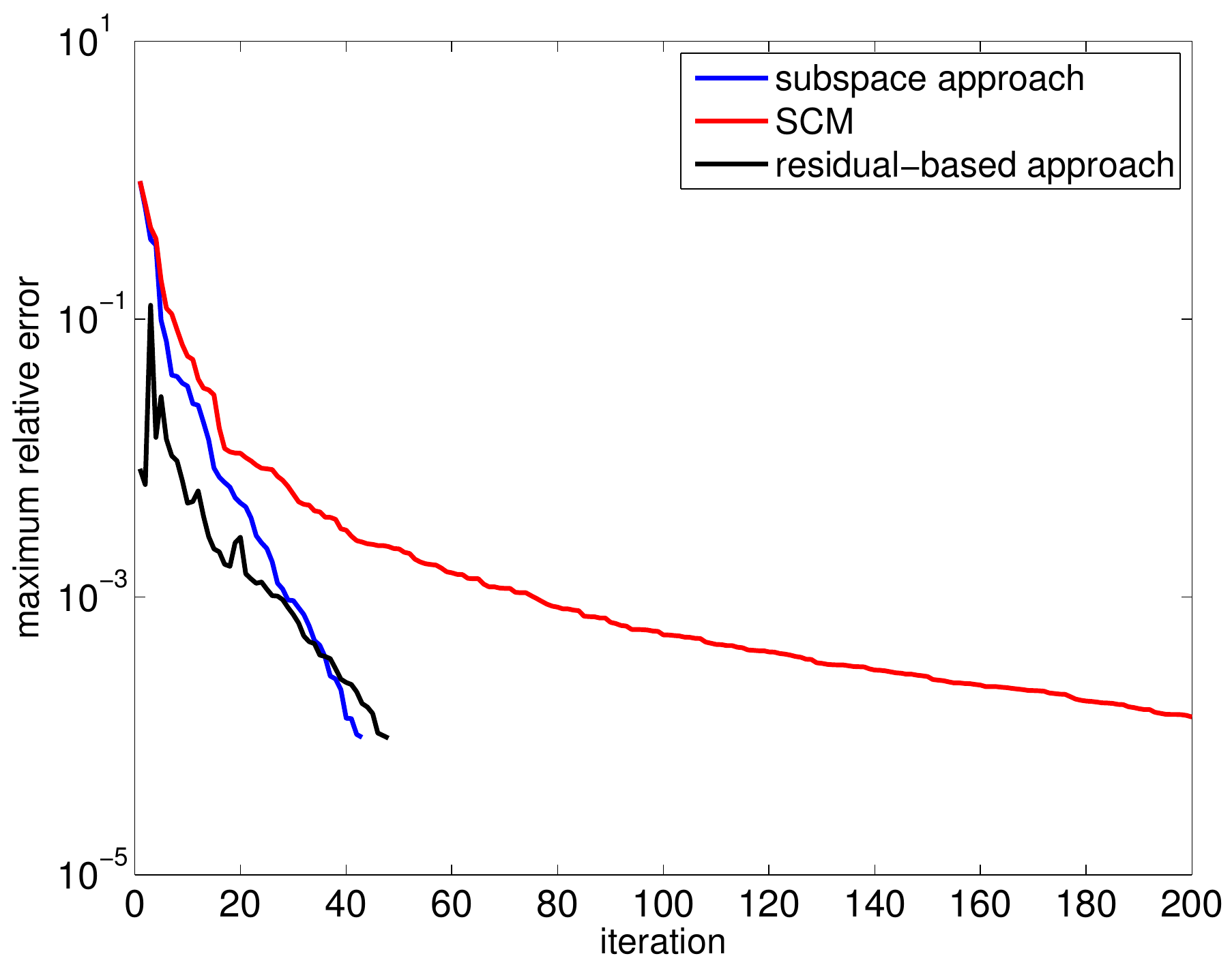}
		\subcaption{Convergence of the maximum relative error ratio~\eqref{eq:error_ratio}.}
		\label{fig:finerrorratioconv}
		\end{subfigure}
		\\
	 	\begin{subfigure}[t]{0.48\textwidth}
		\includegraphics[width=\textwidth]{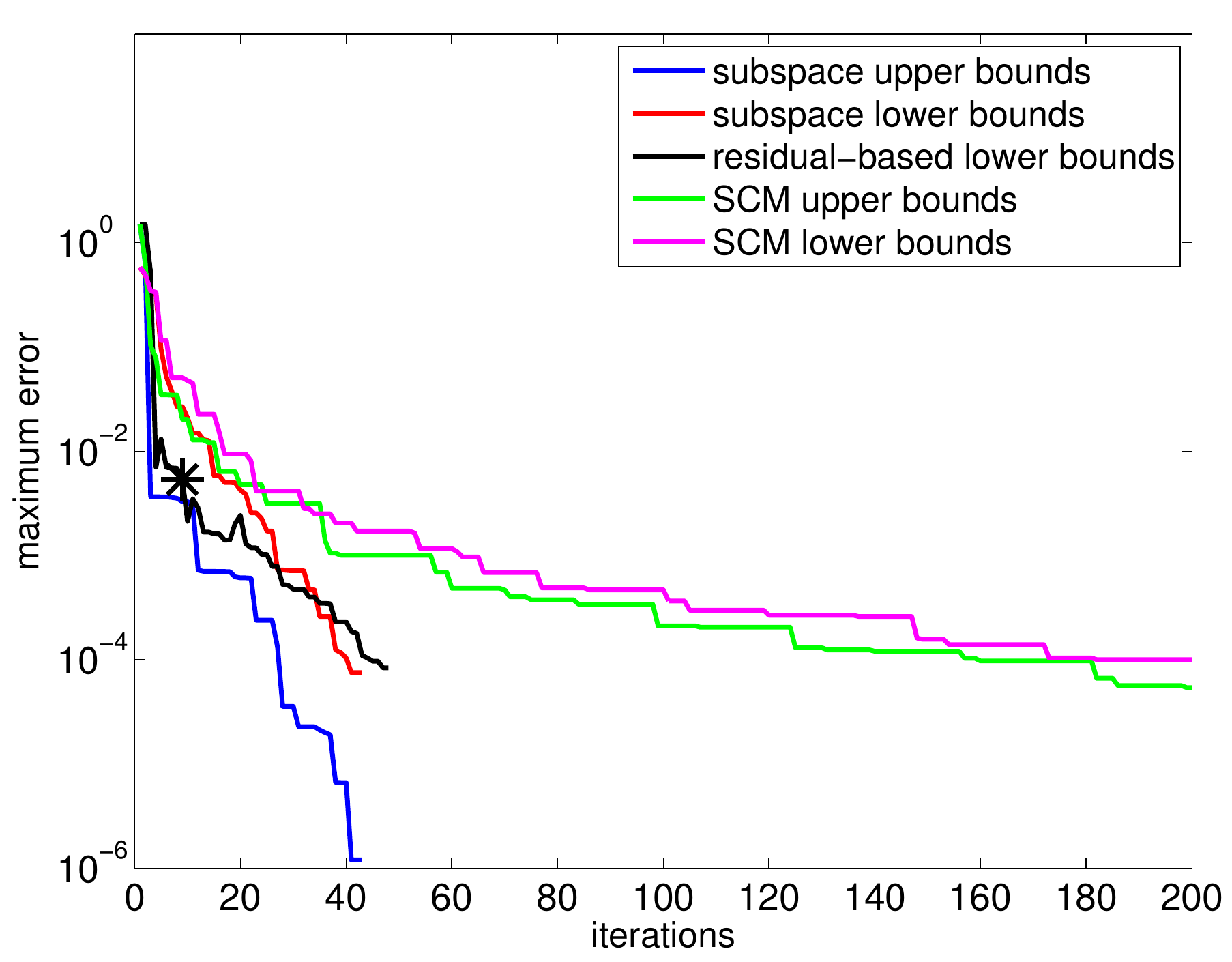}
		\subcaption{Convergence of the error~\eqref{eq:bounderror} for the bounds w.r.t. iteration.}
		\label{fig:finboundconv}
		\end{subfigure}
		~
	 	\begin{subfigure}[t]{0.48\textwidth}
		\includegraphics[width=\textwidth]{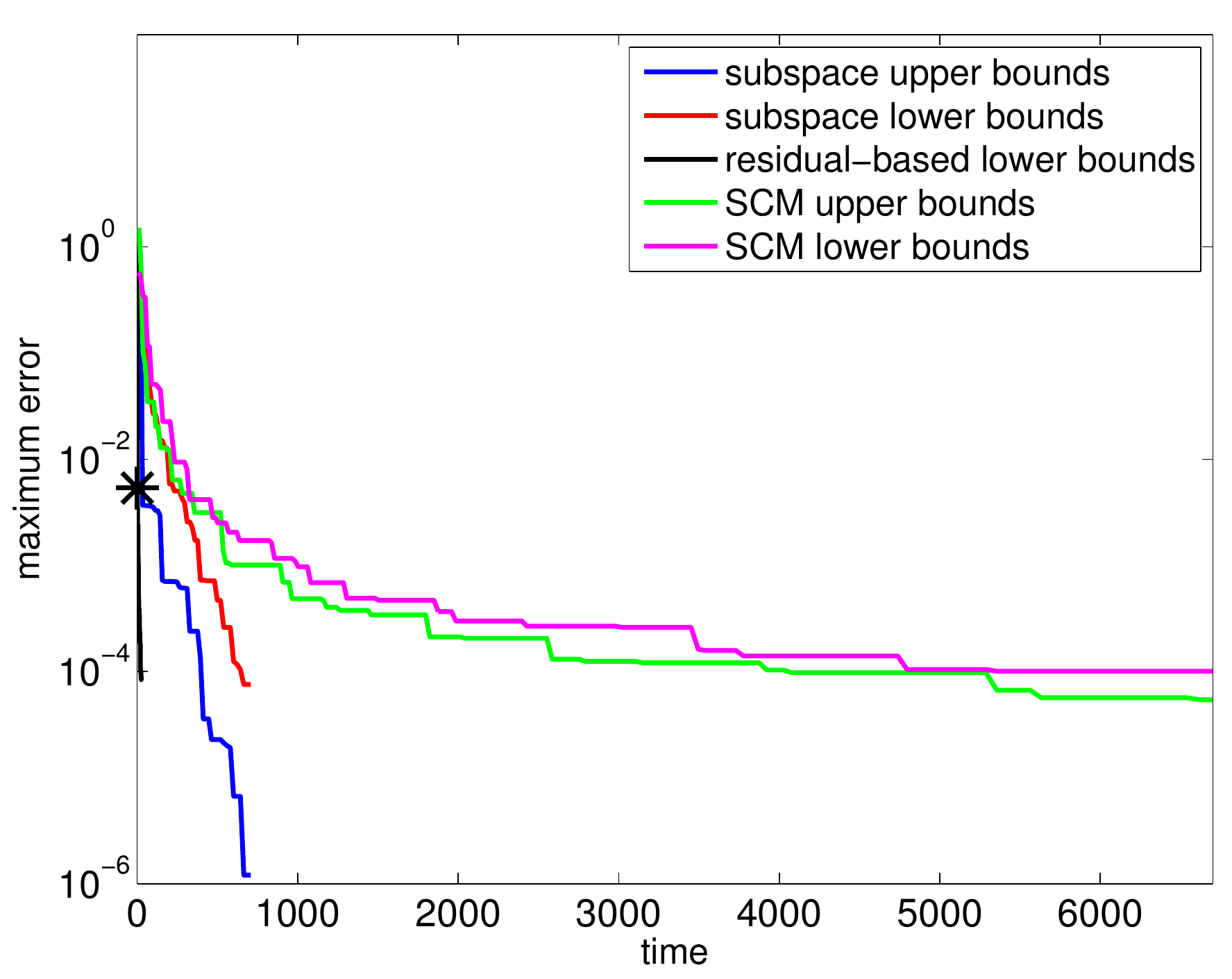}
		\subcaption{Convergence of the error~\eqref{eq:bounderror} for the bounds w.r.t. time.}
		\label{fig:fintimeboundconv}
		\end{subfigure}
	\caption{Convergence plots for Algorithm~\ref{alg:SCM} and~\ref{alg:subspace} applied to Example~\ref{example:fin}.}
	\label{fig:fin}
\end{figure}

\begin{example}\label{example:TBaniso}
	This example concerns a stationary heat equation on a square domain divided into blocks (see Figure~\ref{fig:TBanisogeometry}). 
	In each of the subdomains, one of the parameters $\mu_1, \dots,\mu_9$ determines a coefficient 
	of the PDE
	\begin{equation*}
		\mathrm{div}\left(\begin{bmatrix}1 & -\mu_i \\ -\mu_i & 1\end{bmatrix} \nabla u \right) = 0 \text{ on } \Omega_i, \quad i=1,\dots,9.
	\end{equation*}
	A discretization of the PDE leads to the matrix $A(\mu)=\sum_{i=1}^Q \theta_i(\mu)A_i$, 
	where $Q=10$, $\mu=(\mu_1, \dots,\mu_9)$ and functions $\theta_i(\mu)$ 
	arising from the parametrization of the PDE coefficients. 
	We choose $\calN = 1056$ and $D=[0.1,0.5]^9$. 
	As can be seen in 
	Figure~\ref{fig:TBaniso}, the performance of both Algorithms~\ref{alg:SCM} and~\ref{alg:subspace}
	is not satisfactory, due to the slow convergence of the SCM and subspace lower bounds. Only the subspace upper bounds converges at a satisfactory rate.
	In this example, the residual-based lower bounds clearly show their advantage. 
	They become reliable after only 31 iterations. 
\end{example}

\begin{figure}[ht]
	\resetsubfigs
	\centering
	   \begin{subfigure}[t]{0.48\textwidth}\centering
	    \begin{tikzpicture}[scale=1]
	\tikzstyle{dot}=[draw,shape=circle,fill=black,inner sep=0pt,minimum size=3pt]
	\tikzstyle{edge} = [draw,thick,-,black]
	\tikzstyle{dirichlet} = [edge,blue,dashed]
	
	\begin{scope}
	
		\draw[edge] (0,0) grid (3,3);
		
		\draw (0,0) node[dot,label={[xshift=0, yshift=-25]$(0,0)$}] (p0) {};
		\draw (0,3) node[dot,label={[xshift=0, yshift=0]$(0,1)$}] (p0) {};
		\draw (3,0) node[dot,label={[xshift=0, yshift=-25]$(1,0)$}] (p0) {};
		\draw (3,3) node[dot,label={[xshift=0, yshift=0]$(1,1)$}] (p0) {};
		
		\draw(0.5,0.5) node {$\mu_1$};
		\draw(0.5,1.5) node {$\mu_2$};
		\draw(0.5,2.5) node {$\mu_3$};
		\draw(1.5,0.5) node {$\mu_4$};
		\draw(1.5,1.5) node {$\mu_5$};
		\draw(1.5,2.5) node {$\mu_6$};
		\draw(2.5,0.5) node {$\mu_7$};	
		\draw(2.5,1.5) node {$\mu_8$};
		\draw(2.5,2.5) node {$\mu_9$};
	\end{scope}

\end{tikzpicture}
		\subcaption{Geometry of the underlying PDE.}
		\label{fig:TBanisogeometry}
        \end{subfigure}
        ~
        \begin{subfigure}[t]{0.48\textwidth}
        \includegraphics[width=\textwidth]{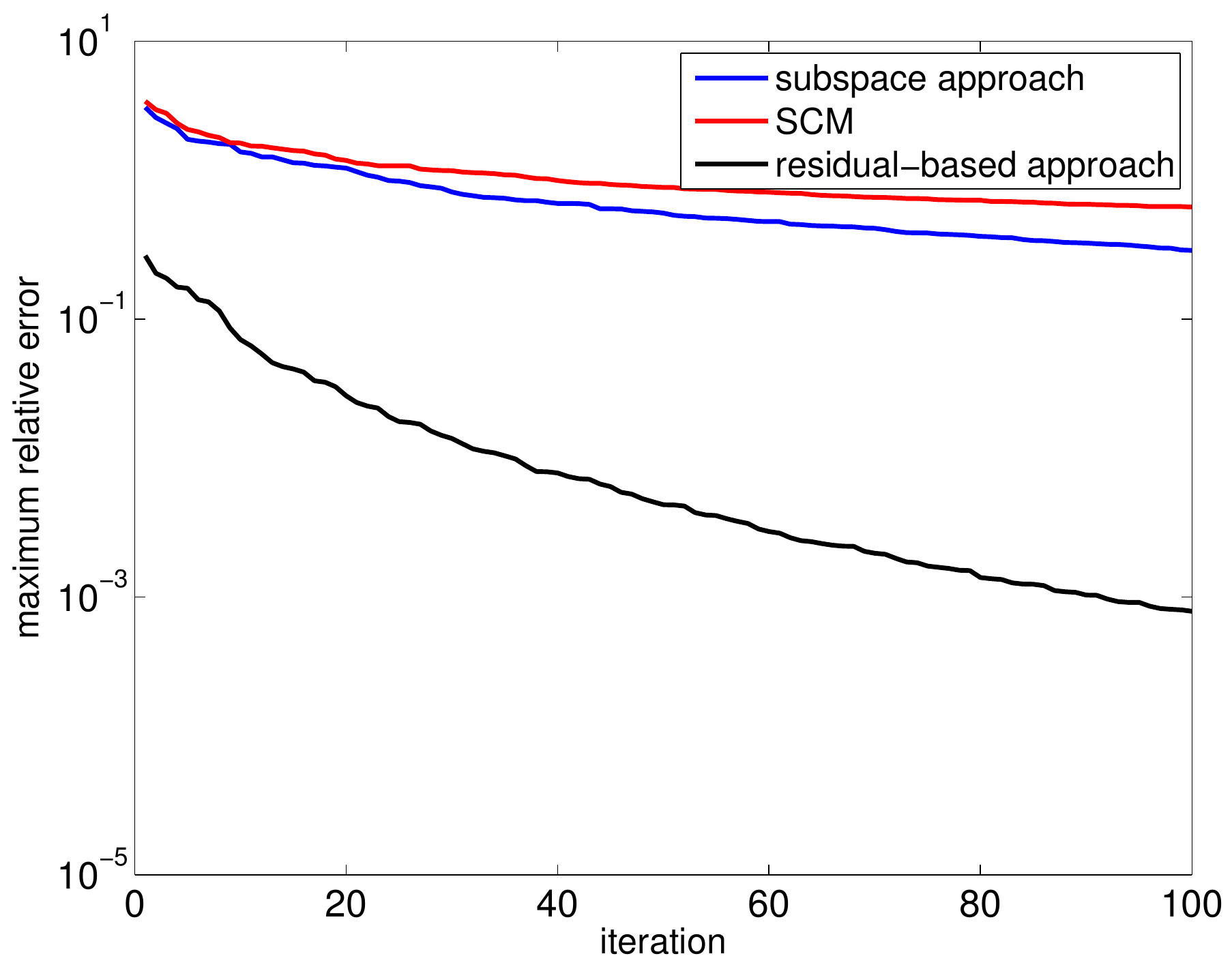}
		\subcaption{Convergence of the maximum relative error ratio~\eqref{eq:error_ratio}.}
		\label{fig:TBanisoerrorratioconv}
		\end{subfigure}
		\\
		\begin{subfigure}[t]{0.48\textwidth}
		\includegraphics[width=\textwidth]{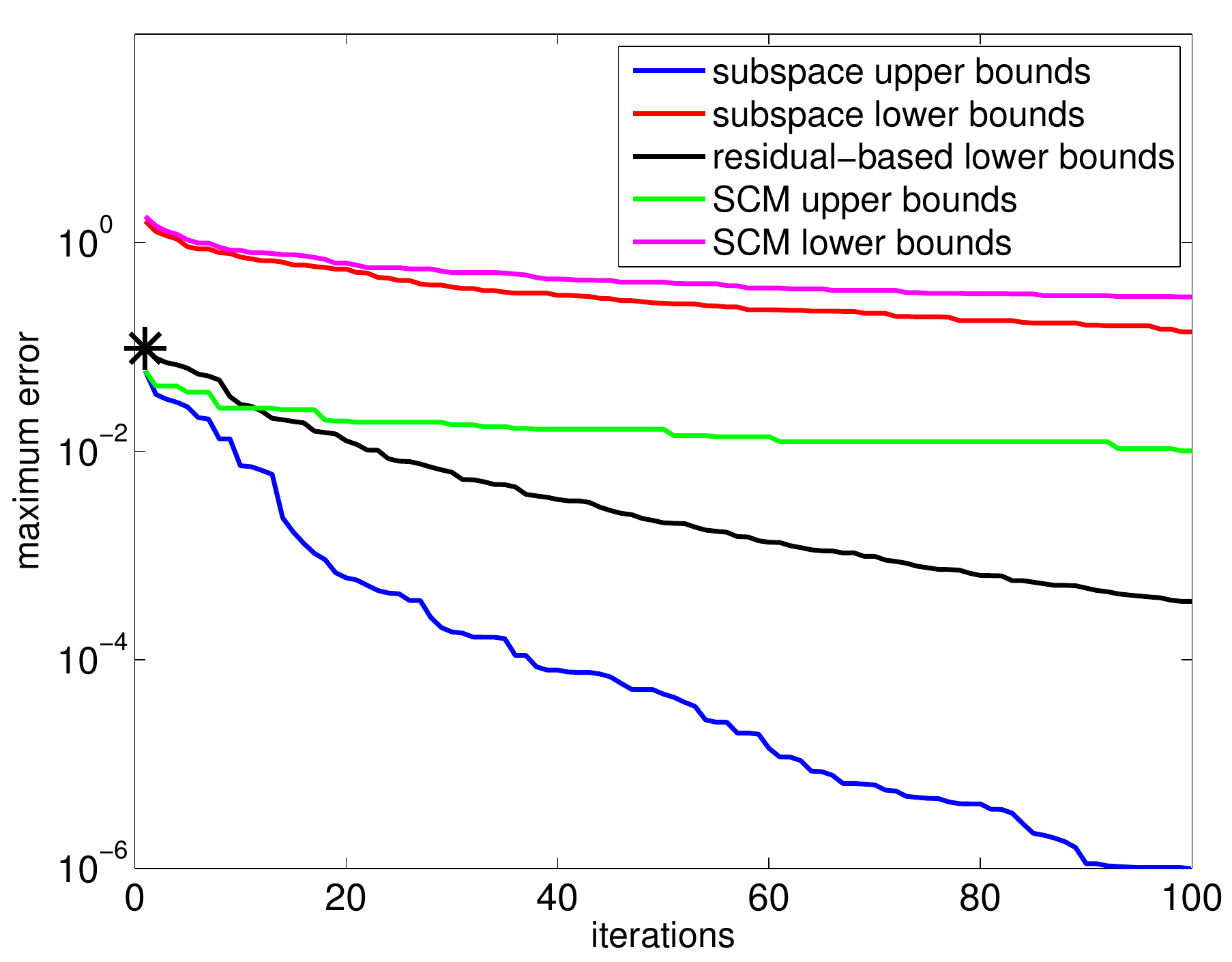}
		\subcaption{Convergence of the error~\eqref{eq:bounderror} for the bounds w.r.t. iteration.}
		\label{fig:TBanisoboundconv}
		\end{subfigure}
        ~
        \begin{subfigure}[t]{0.48\textwidth}
		\includegraphics[width=\textwidth]{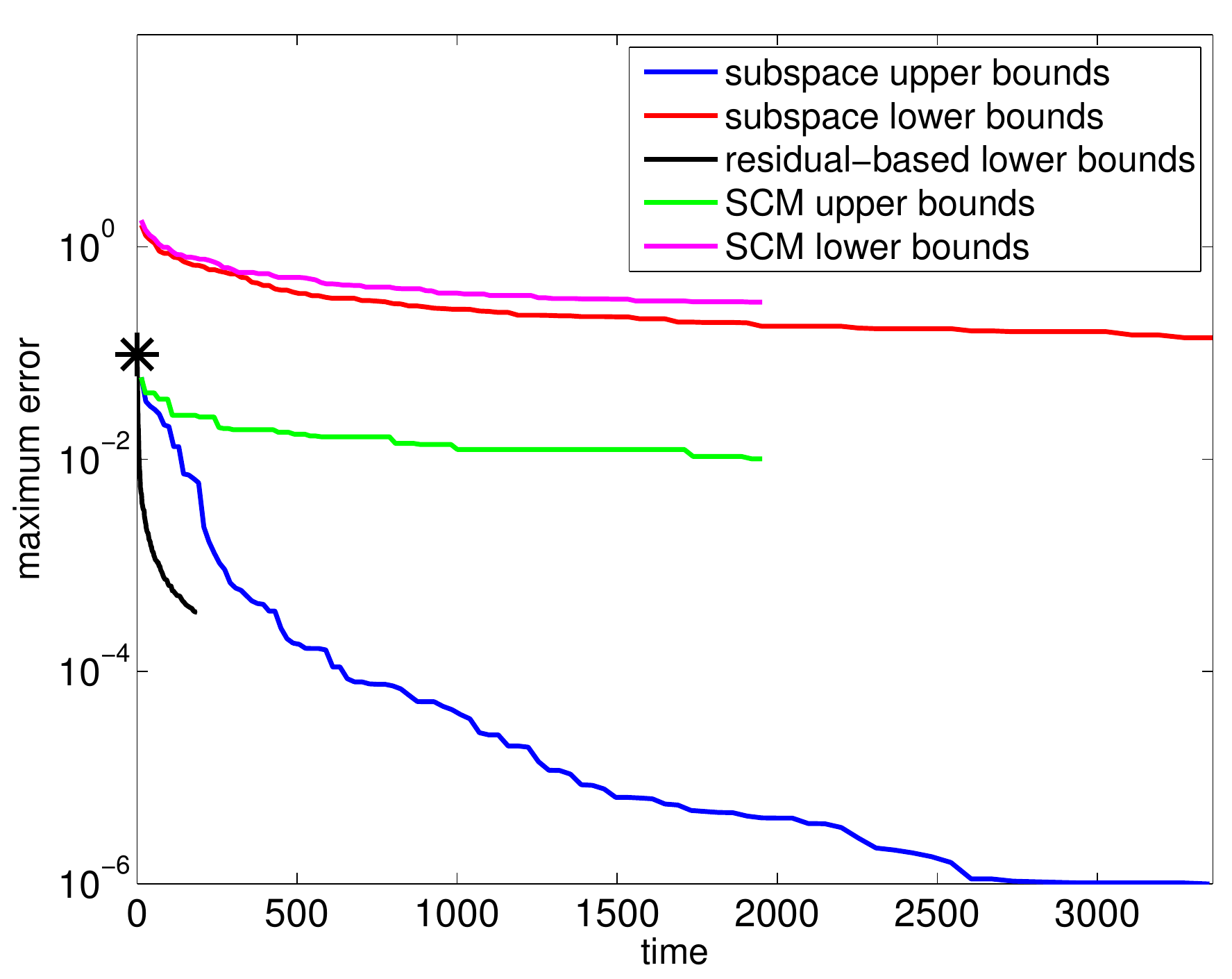}
		\subcaption{Convergence of the error~\eqref{eq:bounderror} for the bounds w.r.t. time.}
		\label{fig:TBanisotimeboundconv}
		\end{subfigure}
	\caption{Convergence plots for Algorithm~\ref{alg:SCM} and~\ref{alg:subspace} applied to Example~\ref{example:TBaniso}.}
	\label{fig:TBaniso}
\end{figure}

\section{Extension to computation of singular values} \label{sec:infsup}

In Section~\ref{sec:coercivity} we have seen that the computation of coercivity constants can be formulated in terms of~\eqref{eq:par_her_eigp}. For non-elliptic parametrized PDE one may have to resort to
the inf-sup constant~\cite{Huynh2010} defined as
\begin{equation}\label{eq:infsup_constant}
	\alpha(\mu) = \inf_{u \in X}\sup_{v \in X} \frac{a(u,v;\mu)}{\|u\|_{X}\|v\|_{X}},
\end{equation}
where $a(\cdot,\cdot,\mu)$ is the bilinear form in the weak formulation of the underlying PDE and $X$.
A finite element discretization of~\eqref{eq:infsup_constant} leads to the minimization problem
\begin{equation} \label{eq:mass_matrix_singvalue}
	\inf_{u\in \R^N}\sup_{v\in \R^N} \frac{u^TA(\mu)v}{\sqrt{u^TXu}\sqrt{v^TXv}}=\inf_{x\in \R^N}\sup_{y\in \R^N} \frac{x^TL^{-T}A(\mu)L^{-1}y}{\|x\|_2\|y\|_2}
\end{equation}
where, once again, $A(\mu)$ and $X=LL^T$ are the discretizations of $a(\cdot,\cdot,\mu)$ and $(\cdot,\cdot)_X$, respectively.
Minimizing~\eqref{eq:mass_matrix_singvalue} is equivalent to solving the singular value problem
\begin{equation*}
	\sigma_{\min}(L^{-1}A(\mu)L^{-T}),
\end{equation*}
which, in turn, is equivalent to computing
\begin{equation}\label{eq:min_sing_value}
	\lambda_{\min}(L^{-1}A(\mu)^TX^{-1}A(\mu)L^{-T}),
\end{equation}
since $\sigma_{\min}(A)=\sqrt{\lambda_{\min}(A^TA)}$. The expression~\eqref{eq:min_sing_value} can be recast in terms of~\eqref{eq:par_her_eigp}, 
with $Q^2$ terms, with the matrices $A_{i,j}$ and functions $\theta_{ij}(\mu)$ for $i,j=1,\dots,Q$  defined as 
\begin{eqnarray*}
	A_{ij} &=& L^{-1}A_i^TX^{-1}A_jL^{-T} \\
	\theta_{ij}(\mu) &=& \theta_i(\mu)\theta_j(\mu).
\end{eqnarray*}
The SCM algorithm has already been applied to~\eqref{eq:min_sing_value} but only with limited success, 
since having $Q^2$ terms in the affine decomposition of $A(\mu)$
further increases the computational cost by making the solution of the LP problem~\eqref{eq:minimization_bound} significantly harder. 
The faster convergence of the subspace-accelerated approach to~\eqref{eq:min_sing_value}
mitigates this cost to a certain extent.



\section{Conclusions} \label{sec:conclusion}

Solving a parametrized Hermitian eigenvalue problem can be computationally very hard and SCM is the most commonly used existing approach. 
We have proposed a new subspace-accelerated approach, given in Algorithm~\ref{alg:subspace}.
As can be seen in Section~\ref{sec:subspace_interpolation}, it has better theoretical properties than SCM. 
As can be seen in Section~\ref{sec:examples}, it also improves significantly on SCM in practice, for a
number of examples discussed in the literature, while having only slightly larger computational cost per 
iteration.
For problems with small gaps between the smallest eigenvalues, as in Example~\ref{example:TBaniso}, the convergence of the subspace lower bounds may still not be satisfactory. 
For such cases, we propose a heuristic approach using residual-based lower bounds.
The proposed approach can be extended to the solution of parametrized singular value problems.

\section{Acknowledgements} \label{sec:acknowledgements}
We thank Christine Tobler and Meiyue Shao for discussing various ideas and approaches with us. 


\appendix
\section{Proof of Lemma~\ref{lemma:function_analysis}} \label{sec:proofs}
As a composition of continuous functions, the function $f$ is clearly continuous. To prove monotonicity we distinguish two cases. First, let $\eta \ge \lambda^{(1)}_{\calV}$. Then
	\[
		f(\eta) = \lambda^{(1)}_{\calV} - 2\rho^2/\Big(\eta-\lambda^{(1)}_{\calV} + \sqrt{(\eta-\lambda^{(1)}_{\calV})^2 + 4\rho^2}\Big),
	\]
	which clearly increases as $\eta$ increases. Now, let $\eta \le \lambda^{(1)}_{\calV}$. Then
	\[
		f(\eta) = \eta - 2\rho^2/\Big(\lambda^{(1)}_{\calV}-\eta + \sqrt{(\eta-\lambda^{(1)}_{\calV})^2 + 4\rho^2}\Big)
	\]
	and
	\[
	 f^\prime(\eta) = 1 - \frac{ 2\rho^2 }{\Big(\lambda^{(1)}_{\calV}-\eta + \sqrt{(\lambda^{(1)}_{\calV}-\eta)^2 + 4\rho^2}\Big)\sqrt{(\lambda^{(1)}_{\calV}-\eta)^2 + 4\rho^2}}.
	\]
	Showing $f^\prime(\eta) \geq 0$, and thus establishing monotonicity, is equivalent to 
	\begin{eqnarray*}
		(\lambda^{(1)}_{\calV}-\eta)^2 + 4\rho^2 + (\lambda^{(1)}_{\calV}-\eta)\sqrt{(\lambda^{(1)}_{\calV}-\eta)^2+4\rho^2} &\geq& 2\rho^2 \\
		(\lambda^{(1)}_{\calV}-\eta)\sqrt{(\lambda^{(1)}_{\calV}-\eta)^2+4\rho^2} &\geq& 0 \geq -(\lambda^{(1)}_{\calV} -\eta)^2 - 2\rho^2, 
	\end{eqnarray*}
	which is trivially satisfied for $\lambda^{(1)}_{\calV} \geq \eta$. This completes the proof.

\bibliographystyle{plain}
\bibliography{anchp}

\end{document}